\documentclass[12pt, letterpaper,table]{amsart}
\usepackage{graphicx} % Required for inserting images
\usepackage{url}
\usepackage{amsmath, amssymb, amsthm}
\usepackage{tikz}
\usepackage{mathdots}
\usepackage{mathtools}
\usepackage{fancybox}
\usetikzlibrary{decorations.pathreplacing,calligraphy}

\newtheorem{theorem}{Theorem}
\newtheorem{definition}[theorem]{Definition}
\newtheorem{lemma}[theorem]{Lemma}
\newtheorem{cor}[theorem]{Corollary}
\newtheorem{conjecture}[theorem]{Conjecture}
\newtheorem{quest}[theorem]{Question}
\newtheorem{prop}[theorem]{Proposition}
\theoremstyle{remark}
\newtheorem{example}[theorem]{Example}
\newtheorem{remark}[theorem]{Remark}
\newtheorem*{remark*}{Remark}
\newtheorem{convention}[theorem]{Convention}
\newtheorem{algorithm}[theorem]{Algorithm}

\newcommand{\fpq}{\frac{p}{q}}
\newcommand{\frs}{\frac{r}{s}}
\newcommand{\calP}{\mathcal{P}}

\newcommand{\stir}{\mathrm{str}}
\newcommand{\calA}{\mathcal{A}}
\newcommand{\calB}{\mathcal{B}}

\newcommand{\highlight}[1]{\colorbox{gray!15}{\small{#1}}}

\definecolor{light-blue}{HTML}{e9f1f8}
\definecolor{grayish}{RGB}{220,220,220}
\definecolor{lavender}{rgb}{0.5,0,1.0}

\usepackage{hyperref}
\hypersetup{
    colorlinks=true,
    linkcolor=lavender,
    filecolor=blue,      
    urlcolor=teal,
    citecolor=teal
}

\usepackage{xcolor}

\definecolor{Dgreen}{RGB}{2,100,64}

\newcommand{\canakci}{{\c{C}}anak{\c{c}}i }
\newcommand{\exx}{\mathrm{ex}}
\newcommand{\wt}{w}

\begin{document}

\title[Orderings on $k$-Markov Numbers]{Orderings on $k$-Markov Numbers}
\author{Esther Banaian}
\email{estherbanaian@gmail.com}
\address{Institut f{\"u}r Matematik \\Universit{\"a}t Paderborn\\Paderborn, Germany}
\date{}

%\date{\today}

\keywords{Markov Numbers, Continued Fractions, Fence Posets, Skein Relations}
%\subjclass{13F55,13F50,13D02}

\begin{abstract}
The $k$-Markov numbers, introduced by Gyoda and Matsushita, are those which appear in positive integral solutions to $x^2 + y^2 + z^2 + k(xy + xz + yz) = (3+3k)xyz$. When $k =0$, this recovers the ordinary Markov numbers. A long-standing question in the theory of Markov numbers is Frobenius's unicity conjecture, concerning whether every Markov number is the maximum in a unique solution triple. Aigner gave a series of weaker, related conjectures which were confirmed to be true by Lee, Li, Rabideau, and Schiffler using techniques from the theory of cluster algebras. We show here that $k$-Markov numbers also satisfy Aigner's conjectures. 
\end{abstract}

\maketitle

\section{Introduction}

A triple of positive integers $(a,b,c) \in \mathbb{Z}_{>0}$ satisfying \[
x^2 + y^2 + z^2 = 3xyz
\]
is known as a \emph{Markov triple}. The equation itself is known as the \emph{Markov equation}. Examples of Markov triples include $(1,1,1), (1,1,2)$, and $(2,5,29)$. A number which appears in a Markov triple is a \emph{Markov number}. For instance, 1,2,5, and 29 are Markov numbers while one can show that 3 is not.

Markov numbers first appeared in the study of Diophantine approximation \cite{markoff1879formes}. In particular, they index points in the discrete interval of the Lagrange spectrum. Since their introduction to the world, Markov numbers have proven to be of interest in many other areas of math, including hyperbolic geometry, algebraic geometry, number theory, and combinatorics. One area of research concerns Frobenius' unicity conjecture, which states that each Markov number is the maximum in a unique Markov triple. As is the case with many long-standing open problems, Frobenius' conjecture has motivated much interesting work on Markov numbers; Aigner's book \cite{aigner2013Markov} provides a lovely summary.

There is a well-known method of indexing Markov numbers with rational numbers in the interval $[0,1]$ (see Section \ref{sec:RationalLabeing}). Formally, we can view this as a function from $\mathbb{Q} \cap [0,1]$ to the set of Markov numbers, sending $\frac{p}{q} \to m_{\frac{p}{q}}$. Frobenius' conjecture is equivalent to stating that this map is injective. If this conjecture is true, then we recover a total order on $\mathbb{Q} \cap [0,1]$ by setting $\fpq \prec \frs$ if $m_\fpq < m_\frs$. Aigner provided a triple of conjectures using this perspective, that is, predicting features of this total ordering.

\begin{conjecture}[Aigner \cite{aigner2013Markov}]\label{conj:Aigner}
Let $p,q$ be positive integers such that $1 \leq p < q$ and $\gcd(p,q) = 1$.
\begin{enumerate}
    \item (Fixed Numerator) If $q < q'$ and $\gcd(p,q') = 1$, then $m_\fpq < m_{\frac{p}{q'}}$.
    \item (Fixed Denominator) If $p < p'< q$ and $\gcd(p',q) = 1$, then $m_\fpq < m_{\frac{p'}{q}}$.
    \item (Fixed Sum) If $0 < i < p$ is such that $\gcd(p-i,q+i) = 1$, then $m_\fpq < m_{\frac{p-i}{q+i}}$.
\end{enumerate}
\end{conjecture}

A modern lens through which some have studied Markov numbers is \emph{cluster theory}. A cluster algebra is a commutative ring with a distinguished set of generators, \emph{cluster variables}, which sit in overlapping sets called \emph{clusters} \cite{fomin2002cluster}. One can already see parallels to the definition of Markov numbers and triples respectively. The similarities run deeper, which the following key result illuminates. 

\begin{theorem}\label{thm:MarkovTreeConnected}
If $(x,y,z)$ is a Markov triple, so is $(x,y,z')$ where $z'$ is defined by \begin{equation}
z' = \frac{x^2 + y^2}{z},\label{eq:Vieta}
\end{equation}
and similarly for $(x',y,z)$ and $(x,y',z)$. Every Markov triple is the result of applying finitely many such moves to the triple $(1,1,1)$. 
\end{theorem}

The positivity of $z'$ evident from Equation \ref{eq:Vieta} while integrality follows from the fact that, if $(x,y,z)$ is a Markov triple, then $\frac{x^2 + y^2}{z} = 3xy-z$. Equation \ref{eq:Vieta} resembles one specific form of \emph{(cluster) mutation}, which connects different clusters in a cluster algebra. Indeed, there is one specific cluster algebra, aptly called the Markov cluster algebra, whose cluster variables specialize to Markov numbers \cite{beineke2011cluster,propp2005combinatorics}. 

This identification of Markov numbers as specializations of cluster variables has been used to make progress towards the unicity conjecture \cite{LLRS,rabideau2020continued}. The Markov cluster algebra is also a cluster algebra of surface type, as discussed in \cite{FST-I,FT-II}; the relevant surface is a once-punctured torus. Musiker, Schiffler, and Williams introduced a direct, combinatorial method of computing cluster variables for cluster algebras of surface type in order to demonstrate the coefficients are always positive \cite{musiker2011positivity}. This construction associates to each cluster variable $x$ a weighted, bipartite graph $\mathcal{G}$ such that $x$ can be expressed as a dimer partition function of $\mathcal{G}$. A corollary in the present setting is that Markov numbers can be interpreted as the number of dimer covers (or, \emph{perfect matchings}) of a family of graphs \cite{propp2005combinatorics}. Indeed, perhaps the earliest instance of a snake graph (predating the definition of cluster algebras) appears in work of Cohn on Markov numbers \cite{Cohn}. These graphs also are closely related to Christoffel words, another combinatorial tool used in the study of Markov numbers \cite{Christoffel}.

Interpreting Markov numbers as cardinalities of sets of dimer covers has been useful in two, related directions. Firstly, \canakci and Schiffler showed that the number of dimer covers of a snake graph can be expressed as the numerator of a \emph{continued fraction} \cite{CSContFrac}. Therefore, one can use machinery related to continued fractions in order to study Markov numbers. (We remark that one could circumvent the cluster algebra connection to link Markov numbers and continued fractions.) This approach was taken by Rabideau and Schiffler in \cite{rabideau2020continued} to prove the Fixed Numerator conjecture (i.e. Conjecture \ref{conj:Aigner} part 1). The second approach uses \emph{snake graph calculus} \cite{CS15,CS13}, a combinatorial method of realizing \emph{skein relations} through the combinatorial formula of Musiker, Schiffler, and Williams \cite{musiker2011positivity}. This approach was taken by Lee, Li, Rabideau, and Schiffler in \cite{LLRS} to show that the three statements in Conjecture \ref{conj:Aigner} follow from a more general set of inequalities on the numbers $m_\fpq$. We note that similar results were shown with other methods in \cite{Gaster,MR4265545,mcshane2021convexity}.  

Here, we will apply both approaches to study  solutions to a family of equations which generalize the Markov equation. Given $k \geq 0$, we define the \emph{$k$-Markov equation} to be  \begin{equation}
x^2 + y^2 + z^2 + k(xy + xz + yz) = (3+3k)xyz.\label{eq:GM_General}
\end{equation}

Triples of positive integers satisfying Equation \ref{eq:GM_General} are called $k$-Markov triples, which are comprised of $k$-Markov numbers. Notice that ordinary Markov triples  coincide with 0-Markov triples in our language, allowing us to simultaneously discuss the ordinary and generalized case. 

Equation \ref{eq:GM_General} was first defined and studied by Gyoda and Matsushita \cite{GyodaMatsushita}. In particular, they showed that Vieta jumps connect all $k$-Markov triples.
\begin{theorem}[Theorem 1.1 \cite{GyodaMatsushita}]\label{thm:GM}
If $(x,y,z)$ is a $k$-Markov triple, so is $(x,y,z')$ where $z'$ is given by \begin{equation}
z' = \frac{x^2 + kxy + y^2}{z} \label{eq:GenVieta}
\end{equation}
and similarly for $(x',y,z)$ and $(x,y',z)$. Every $k$-Markov triple is the result of applying finitely many such moves to the triple $(1,1,1)$. 
\end{theorem}

 The definition of the $k$-Markov equation is motivated by Chekhov and Shapiro's \emph{generalized cluster algebras}, which have the same structure as (ordinary) cluster algebras but allow more general mutation relations. Indeed, for each choice of $k$, there is a specific generalized cluster algebra whose cluster variables specialize to $k$-Markov numbers. 
 
 We remark that  Gyoda and Matsushita studied the more general equation \begin{equation}
x^2 + y^2 + z^2 + k_1yz + k_2 xz + k_3 xy = (3+k_1 + k_2 + k_3)xyz
 \label{eq:DistinctKs}\end{equation}
with three nonnegative integral parameters $k_1,k_2,k_3$ \cite{GyodaMatsushita} . An appropriately defined analogue of Theorem \ref{thm:GM} is true here and the connection to generalized cluster algebras remains. However, the uniqueness conjecture is known to not be true in this widest setting of generality (see \cite[Remark 2.8]{GyodaMatsushita}), and the asymmetry forces one to use a more complex indexing system (see \cite[Section 3.4]{BG}). These two complications convince us to only consider the case $k_1 = k_2 = k_3$ here. However, we remark that the results in Sections \ref{sec:TwoPosets} and \ref{sec:MarkovDistance} could be extended to the distinct $k_i$ case. 

Our main result is that Aigner's conjectures holds for $k$-Markov numbers as well.

\begin{theorem}\label{thm:Main}
Conjecture \ref{conj:Aigner} holds for the $k$-Markov numbers for all $k \geq 0$.
\end{theorem}

 Theorem \ref{thm:GM} for $k = 1$ (i.e. 1-Markov numbers) was independently shown by the author and Sen in \cite{banaian2024generalization}. In this case, the generalized cluster algebra can be seen as arising from a once-punctured sphere with three orbifold points of order 3. The main focus in this previous work was to use a generalization of the snake graph construction to orbifolds from previous work with Kelley \cite{BanaianKelley} to study 1-Markov numbers, largely through continued fractions.

In joint work with Gyoda \cite{BG}, we repackaged the construction from \cite{BanaianKelley} into the language of \emph{fence posets} and used it to study  the generalized cluster algebras in the backdrop of the $k$-Markov numbers and the more general solutions to Equation \ref{eq:DistinctKs}. These posets come with labels and weights which are scalar multiples of monomials; we can use this construction to study the $k$-Markov numbers as well by setting the initial variables to 1 in the weights and labels. The resulting posets can still come with weights other than 1, so that we are no longer interpreting $k$-Markov numbers as cardinalities. However, we use a simple trick (Lemma \ref{lem:ExtendPoset}) to ``extend'' the weighted posets to those with all elements with weight 1. Our main proof technique is to apply the poset skein relations from joint work with Kang and Kelley \cite{banaian2024skein} to either the weighted posets or these extended posets. In Section \ref{sec:Skein}, we highlight how these poset skein relations can be used to deduce relations on numerators of continued fractions, which may be of independent interest.

The remainder of the paper is structured as follows. In Section \ref{sec:Tools}, we complete our survey of the background as well as give more precise details regarding the combinatorial tools discussed here. The poset construction from \cite{BG} (in the specialized case) as well as the extended version is given in Section \ref{sec:TwoPosets}. Section \ref{sec:MarkovDistance} describes a distance function we put on $\mathbb{Z}^2$, $\ell_k$,  using the combinatorial constructions in the previous section and following the ordinary case from \cite{LLRS}. Finally, Section \ref{sec:InducedOrdering} contains the proof of the main result, using properties of $\ell_k$, and culminates with a discussion regarding future directions.

\section*{Acknowledgements}

The author is deeply grateful to Yasuaki Gyoda for providing the initial inspiration for this work and valuable feedback, including pointing out a hole in the proof of Proposition \ref{prop:StraightenedIsShortest} in an earlier version.

\section{Tools}\label{sec:Tools}

In this section, we summarize necessary background for the main result. Many topics are inspired by cluster theory, which we point out for continuity and motivation. However, knowledge of cluster algebras is not necessary to understand any results in this article.

\subsection{Rational Labeling}\label{sec:RationalLabeing}

Theorem \ref{thm:GM} endows $k$-Markov triples with the structure of a rooted tree. It is most convenient to omit the first triple, $(1,1,1)$, from the construction.

\begin{definition}
Define the \emph{$k$-Markov tree}, $\mathrm{MT}^k$, to be the rooted tree whose vertices are labeled by triples of positive integers with root $(1,k+2,1)$ and such that each vertex $(a,b,c)$ has left child $(b,\frac{b^2 + kbc + c^2}{a},c)$ and right child $(a, \frac{a^2 + kab + b^2}{c},c)$.

%Define also the \emph{principal $k$-Markov tree}, $\mathrm{T}^k$, to be the rooted subtree of $\mathrm{MT}^k$ with root given by the right child of $(1,k+2,1)$. 
\end{definition}

%We define the principal $k$-Markov tree since otherwise there is repetition within the tree. 
  We draw a neighborhood of the root of the 1-Markov tree below. Notice that there is symmetry between the left and right side of the tree, up to reordering.
%The principal 1-Markov tree is on the right.

\begin{center}
\begin{tikzpicture}[scale=0.65]
\node(a) at (0,10){$(1,3,1)$};
\node(b) at (-6,9){$(3, 13, 1)$};
\node(c) at (6,9){$(1,13,3)$};
\draw(a) -- (b);
\draw(a) -- (c);
\node(d) at (-8.5,8){$(13,61,1)$};
\node(e) at (-2.5,8){$(3,217,13)$};
\draw(b) -- (d);
\draw(b) -- (e);
\node(f) at (2.5,8){$(13,217,3)$};
\node(g) at (7.5,8){$(1,61,13)$};
\draw(c) -- (f);
\draw(c) -- (g);
\node(h) at (-10,7){$(61,291,1)$};
\node(i) at (-7.5,6){$(13,4683,61)$};
\draw(d) -- (h);
\draw(d) -- (i);
\node(j) at (-4.5,7){$(217,16693,13)$};
\node(k) at (-1.5,6){$(3,3673,217)$};
\draw(e) -- (j);
\draw(e) -- (k);
\node(l) at (1.5,7){$(217,3673,3)$};
\node(m) at (4.5,6){$(13,16693,217)$};
\draw(f) -- (l);
\draw(f) -- (m);
\node(n) at (7.5,7){$(61,4683,13)$};
\node(o) at (10,6){$(1,291,61)$};
\draw(g) -- (n);
\draw(g) to [out = 0, in = 90] (o);
\end{tikzpicture}
\end{center}

We will label $k$-Markov numbers with rational numbers by comparing the $k$-Markov tree with the Farey tree.

\begin{definition}
Define the \emph{Farey tree}, $\mathrm{FT}$, to be the rooted tree whose vertices are labeled with triples from $\mathbb{Q} \cup \{\frac10\}$ with root $(\frac01, \frac11, \frac10)$ and such that each vertex $(\frs, \fpq, \frac{t}{u})$ has left child $(\fpq, \frac{p+t}{q+u}, \frac{t}{u})$ and right child $(\frs, \frac{r+p}{s+q},\fpq)$.

%Define also the \emph{small Farey tree}, $\mathrm{FT}_{(0,1)}$ to be the rooted subtree of $\mathrm{FT}$ with root given by $(\frac01, \frac12, \frac11)$, i.e., the right child of $(\frac01, \frac11, \frac10)$.
\end{definition}

We draw a neighborhood of the root of the Farey tree below. Notice that the rooted sub-tree with root $(\frac01, \frac12, \frac11)$ is labeled by triples coming from the interval $[0,1]$. Let $\mathbb{Q}_{[0,1]}:= \mathbb{Q} \cap [0,1]$.

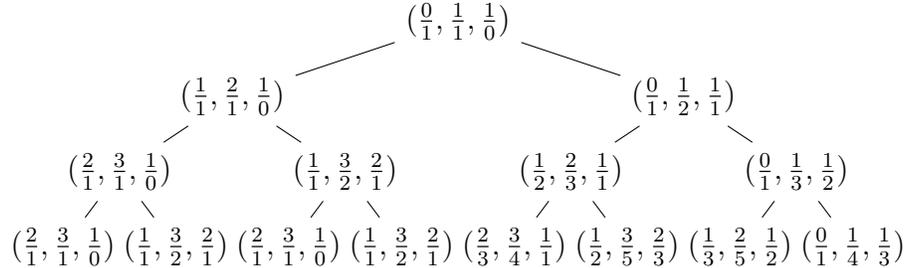
\begin{center}
\begin{tikzpicture}
\node(a) at (0,10){$(\frac01,  \frac11, \frac10)$};
\node(b) at (-3,9){$(\frac11,  \frac21, \frac10)$};
\node(c) at (3,9){$(\frac01,  \frac12, \frac11)$};
\draw(a) -- (b);
\draw(a) -- (c);
\node(d) at (-4.5,8){$(\frac21, \frac31, \frac10)$};
\node(e) at (-1.5,8){$(\frac11, \frac32, \frac21)$};
\draw(b) -- (d);
\draw(b) -- (e);
\node(f) at (1.5,8){$(\frac12, \frac23, \frac11)$};
\node(g) at (4.5,8){$(\frac01, \frac13, \frac12)$};
\draw(c) -- (f);
\draw(c) -- (g);
\node(h) at (-5.25,7){$(\frac21, \frac31, \frac10)$};
\node(i) at (-3.75,7){$(\frac11, \frac32, \frac21)$};
\draw(d) -- (h);
\draw(d) -- (i);
\node(j) at (-2.25,7){$(\frac21, \frac31, \frac10)$};
\node(k) at (-0.75,7){$(\frac11, \frac32, \frac21)$};
\draw(e) -- (j);
\draw(e) -- (k);
\node(l) at (5.25,7){$(\frac01, \frac14, \frac13)$};
\node(m) at (3.75,7){$(\frac13, \frac25, \frac12)$};
\draw(g) -- (l);
\draw(g) -- (m);
\node(n) at (2.25,7){$(\frac12, \frac35, \frac23)$};
\node(o) at (0.75,7){$(\frac23, \frac34, \frac11)$};
\draw(f) -- (n);
\draw(f) -- (o);
\end{tikzpicture}
\end{center}

%The definition of the small Farey tree is designed to mimic the principal $k$-Markov tree. 
It is well-known that every positive rational number appears as the middle entry of a unique vertex in $\mathrm{FT}$.
%, and in particular every rational number in $\mathbb{Q} \cap (0,1)$ appears as the  middle entry of a unique vertex in the small Farey tree.

There is a canonical bijection of trees $\overline{\iota}: \mathrm{FT} \to \mathrm{MT}^k$. We use this to define a map  $\iota:  \mathbb{Q}_{[0,1]}\to \mathbb{Z}$ by focusing on the middle entries of each, \[
\iota\bigg(\fpq\bigg) = b \text{ where } \overline{\iota}\bigg(\frs, \fpq, \frac{t}{u}\bigg) = (a,b,c).\] Using this comparison of trees, given $\fpq \in \mathbb{Q}_{[0,1]}$, we define $m^{(k)}_{\fpq}$ as \[
m^{(k)}_\fpq := \begin{cases} 1 & \fpq = \frac01\\
k+2 & \fpq = \frac11\\
\iota(\fpq) & 0 < \fpq < 1.\end{cases}
\]

For example, $m_{\frac25}^{(1)} = 4683$.
In Section \ref{sec:TwoPosets}, we will give an explicit method to construct $m^{(k)}_\fpq$.

\subsection{Continued Fractions}

This section discusses an important computational tool which is intricately linked to both Markov numbers and fence posets (to be discussed in the next section).

\begin{definition}
Given a finite list $a_1,\ldots,a_n$ of nonnegative integers, with $a_i > 0$ for $i > 1$,  the \emph{continued fraction}, $[a_1,\ldots,a_n]$ is given by \[
[a_1,\ldots,a_n] = a_1 + \cfrac{1}{a_2 + \cfrac{1}{a_3 + \cfrac{1}{\ddots + \cfrac{1}{a_n}}}}.
\]
\end{definition}

Every rational number $\fpq$ can be expressed as a continued fraction; this expression is unique if we insist $a_n > 1$. It will be advantageous for us later to compute continued fractions in an alternate way.

\begin{lemma}\label{lem:ComputeWithMatrix}
Let $a_1,\ldots,a_n$ be a list of nonnegative integers with $a_n > 0$. If $p,q,r,s$ are integers defined by \[
\begin{pmatrix} a_1 & 1 \\ 1 & 0 \end{pmatrix} \begin{pmatrix} a_2 & 1 \\ 1 & 0 \end{pmatrix} \cdots \begin{pmatrix} a_n & 1 \\ 1 & 0 \end{pmatrix} = \begin{pmatrix} p & r\\ q & s\end{pmatrix},
\]
then \[
[a_1,\ldots,a_n] = \fpq \qquad \text{and} \quad [a_1,\ldots,a_{n-1}] = \frs.
\]
\end{lemma}

One can replace the definition of a continued fraction with the matrix formula in Lemma \ref{lem:ComputeWithMatrix}, in which case it is no longer problematic if $a_i = 0$. Next, we use Lemma \ref{lem:ComputeWithMatrix} to relate numerators of continued fractions. These formulas will be useful in later sections. Let $\mathcal{N}[a_1,\ldots,a_n]$ denote the numerator of the resulting (reduced) fraction $[a_1,\ldots,a_n]$. That is, if $[a_1,\ldots,a_n] = \fpq$, then $\mathcal{N}[a_1,\ldots,a_n] = p$.  

\begin{lemma}\label{lem:ContFracSkein}
Let $\mu_1,\mu_2$ be possibly empty lists of nonnegative integers and let $a,b,k$ be nonnegative integers. We have \begin{equation}
\mathcal{N}[\mu_1,a,k,b,\mu_2] = \mathcal{N}[\mu_1,a+k+b,\mu_2] + k\mathcal{N}[\mu_1,a-1,1,b-1,\mu_2]\label{eq:ContFracSkein1}
\end{equation}
and 
\begin{equation}
\mathcal{N}[\mu_1,a,k+1,b,\mu_2] = \mathcal{N}[\mu_1,a,1,b+k,\mu_2] + k\mathcal{N}[\mu_1,a-1,1,b-2,\mu_2]\label{eq:ContFracSkein2}
\end{equation}
\end{lemma}

\begin{proof}
In light of Lemma \ref{lem:ComputeWithMatrix}, we can show each equation by comparing two matrix expressions. For Equation \ref{eq:ContFracSkein1}, first we compute \[
\begin{pmatrix} a & 1 \\ 1&0 \end{pmatrix} \begin{pmatrix} k & 1 \\ 1&0 \end{pmatrix}\begin{pmatrix} b & 1 \\ 1&0 \end{pmatrix} = \begin{pmatrix} akb + a + b & ak+1 \\ bk+1 &k \end{pmatrix}.
\]
Next, we see 
\begin{equation}
\begin{pmatrix} k & 0 \\ 0 & k \end{pmatrix}  \begin{pmatrix} a-1 & 1 \\ 1&0 \end{pmatrix} \begin{pmatrix} 1 & 1 \\ 1&0 \end{pmatrix}\begin{pmatrix} b-1 & 1 \\ 1&0 \end{pmatrix} = \begin{pmatrix} akb - k & ak \\ bk & k \end{pmatrix}\label{eq:matrix}
\end{equation}
and the entry-wise sum of the above with \[
\begin{pmatrix} a+k+b & 1 \\ 1&0 \end{pmatrix} \]
is the same as the matrix in Equation \ref{eq:matrix}, showing Equation \ref{eq:ContFracSkein1}. One can show Equation \ref{eq:ContFracSkein2} in an analogous way.
\end{proof}

One more useful result is the following, which is well-known.

\begin{lemma}\label{lem:Reverse}
If $a_1,\ldots,a_n$ is a finite sequence of positive integers, then \[
\mathcal{N}[a_1,\ldots,a_n] = \mathcal{N}[a_n,\ldots,a_1].
\]
\end{lemma}

\subsection{Snake Graphs and Fence Posets}

Here, we introduce two combinatorial objects which have parallel roles in the theory of cluster algebras from surfaces. 

A \emph{tile} is another name for a cycle graph on 4 vertices. We will draw tiles as squares and will refer to the four sides as North, East, South, and West in the natural way. If the tile is named $G$, then the north edge will be labeled $N(G)$, and similarly for the other cardinal directions.

Consider a set of tiles, $G_1,\ldots,G_d$. A \emph{snake graph}, $\mathcal{G}$, is the result of identifying pairs of edges of these graphs in such a way that, for all $1 \leq i < d$, either $N(G_i)$ is identified with $S(G_{i+1})$ or $E(G_i)$ is identified with $W(G_{i+1})$. The resulting shape is also known as a \emph{border-strip}. An example is drawn in Figure \ref{fig:SnakeGraphAndPosetExample} on the left. 

\begin{figure}
\centering
\begin{tikzpicture}
\draw(2,3) -- (1,3) -- (1,1) -- (0,1) -- (0,0) -- (2,0) -- (2,2) -- (4,2) -- (4,3) -- (2,3) ;
\draw(1,0) -- (1,1) -- (2,1);
\draw(1,2) -- (2,2) -- (2,3);
\draw(3,3) -- (3,2);
\node[] at (0.5,0.5){$G_1$};
\node[] at (1.5,0.5){$G_2$};
\node[] at (1.5,1.5){$G_3$};
\node[] at (1.5,2.5){$G_4$};
\node[] at (2.5,2.5){$G_5$};
\node[] at (3.5,2.5){$G_6$};
\draw(7,3) to node[below, scale = 0.75]{$-$} (6,3) to node[right,scale=0.75]{$-$} (6,2) to node[right,scale=0.75]{$+$}  (6,1) to node[below,scale=0.75]{$+$} (5,1) to node[right,scale=0.75]{$+$} (5,0) to node[above, scale = 0.75, red]{$-$} (6,0) to node[above, scale = 0.75]{$+$} (7,0) to node[left, scale = 0.75]{$+$} (7,1)to node[left, scale = 0.75]{$-$}  (7,2) to node[above,scale=0.75]{$-$} (8,2)  to node[above,scale=0.75]{$+$} (9,2)  to node[left,scale=0.75]{$+$} (9,3)  to node[below,scale=0.75]{$-$} (8,3) to node[below,scale=0.75]{$+$}  (7,3);
\draw(6,0) to node[right,scale=0.75,red]{$-$}  (6,1) to node[above,scale=0.75,red]{$-$} (7,1);
\draw(6,2) to node[above,scale=0.75,red]{$+$} (7,2) to node[right,scale=0.75,red]{$+$} (7,3);
\draw(8,3) to node[right,scale=0.75,red]{$-$} (8,2);
\node(1) at (10,0.5){$\boxed{1}$};
\node(2) at (11,1.5){$\boxed{2}$};
\node(3) at (12,2.5){$\boxed{3}$};
\node(4) at (13,1.5){$\boxed{4}$};
\node(5) at (14,0.5){$\boxed{5}$};
\node(6) at (15,1.5){$\boxed{6}$};
\draw(1) -- (2);
\draw(2) -- (3);
\draw(3) -- (4);
\draw(4) -- (5);
\draw(5) -- (6);
\end{tikzpicture}
\caption{On the left, we have a snake graph with 6 tiles. The sign function on the snake graph is shown in the middle, with the signs used to compute the shape colored in red. The red signs (doubling the last $-$) tell us that the snake graph has shape $3,2,2$. The fence poset of shape $3,2,2$ is drawn on the right. Here, we label elements with their chronological labels in \boxed{boxes}.}\label{fig:SnakeGraphAndPosetExample}
\end{figure}
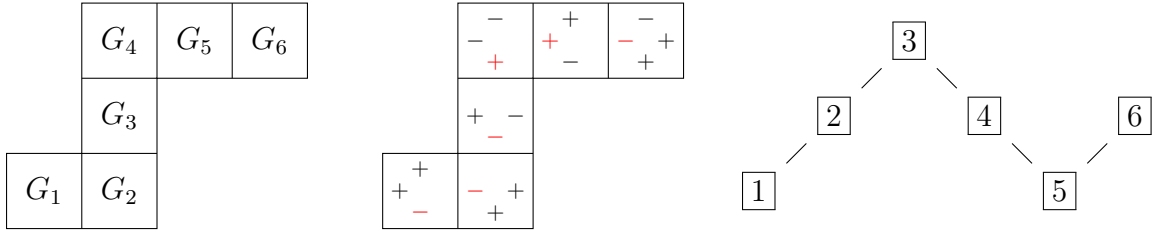

There are many ways to glue a set of tiles together. In order to distinguish these, we use a \emph{sign function}, $f$, as in \cite{CSContFrac}. This is a function from the edge set of a snake graph to the size two set $\{\pm\}$, which satisfies the following rules.
\begin{enumerate}
    \item For each tile $G_i$, $S(G_i)$ and $E(G_i)$ are labeled the same way and similarly for $N(G_i)$ and $W(G_i)$.
    \item For each tile $G_i$, $S(G_i)$ and $N(G_i)$ are labeled in the opposite way.
    \item The label of $S(G_1)$ is $-$.
\end{enumerate}

The final condition assigns a unique sign function to each snake graph; otherwise, there would be two sign functions associated to each snake graph. See the middle of Figure \ref{fig:SnakeGraphAndPosetExample} for an example of a sign function.

An \emph{internal edge} of a snake graph is one which borders two tiles. A snake graph with $d$ tiles has $d-1$ internal edges. Label these $e_1,\ldots,e_{d-1}$. Let $e_0$ denote $S(G_1)$. The \emph{sign sequence} of a snake graph is $f(e_0),f(e_1),\ldots,f(e_{d-1}),f(e_{d-1})$. We emphasize that we repeat $f(e_{d-1})$ at the end. Let $a_1,\ldots,a_n$ denote the lengths of maximal constant subseqeunces of the sign sequence; necessarily, $a_n > 1$. We call the sequence $a_1,\ldots,a_n$ the \emph{shape} of of $\mathcal{G}$. Indeed the data of the shape uniquely determines $\mathcal{G}$. For instance, the shape of the snake graph in Figure \ref{fig:SnakeGraphAndPosetExample} is 3,2,2.

Recall a \emph{perfect matching} of a graph $G = (V,E)$ is a $M \subseteq E$ such that every $v \in V$ is incident to exactly one $e \in M$. Weighted perfect matchings of snake graphs can be used to compute cluster variables in a cluster algebra of surface type \cite{musiker2011positivity}. The following beautiful result by \canakci and Schiffler shows how the shape of a snake graph also encodes the number of perfect matchings. 

\begin{theorem}[\cite{CSContFrac}]\label{thm:CountPMSWithContFrac}
If $\mathcal{G}$ is a snake graph of shape $a_1,\ldots,a_n$, then the number of perfect matchings of  $\mathcal{G}$ is $\mathcal{N}[a_1,\ldots,a_n]$.
\end{theorem}

For example, the continued fraction $[3,2,2]$ is equal to the rational number $\frac{17}{5}$, from which we see that the snake graph in Figure \ref{fig:SnakeGraphAndPosetExample} has 17 perfect matchings.

Cluster variables are Laurent polynomials in two variables. Computing one of these variables in terms of snake graphs can be done by putting the set of perfect matchings into a partial order. The resulting poset turns out to be a distributive lattice, and its underlying poset of join-irreducible elements is a \emph{fence poset} \cite[Theorem 5.4]{musiker2013bases}. A fence poset is a poset whose Hasse diagram is a path graph. In other words, a fence poset $\calP$ on $h$ elements, denoted $\calP(1),\ldots,\calP(h)$, is a poset which is either of the form \begin{equation}\calP(1)\prec \calP(2) \prec \cdots \prec \calP(c_1) \succ \calP(c_1 + 1) \succ \cdots \succ \calP(c_2) \prec \calP(c_2+1 )\cdots \calP(c_n) = \calP(h),\label{eq:FencePosetDef} \end{equation}
or the result of flipping every relation.

Notice that by our indexing, an element $\calP(i)$ only has cover relations with $\calP(i-1)$ and $\calP(i+1)$, when these elements exist. We refer to this as a \emph{chronological labeling}, as in \cite{banaian2024skein}. See the fence poset on the right in Figure \ref{fig:SnakeGraphAndPosetExample}. We use boxes to denote the use of chronological labeling and to distinguish these from other labels which will be used in later sections.

Every poset with more than one element has two chronological labelings. Given a poset $\calP$ with one chronological labeling, let $\overline{\calP}$ denote the same poset with reverse chronological labeling. In a drawing, this means taking a mirror image of the Hasse diagram. The notion of chronological labeling  also invites a notion of interval notation, $\calP[a,b] := \{\calP(i): a \leq i \leq b\}$. If $a > b$, define $\calP[a,b] = \emptyset$.

We return to discussions of shape. If $n$, as in Equation \ref{eq:FencePosetDef}, is 1, so that  $\calP$ is a chain which increases with the chronological labeling, we define the \emph{shape} of $\calP$ to be $a_1 = c_1+1$. Otherwise, we define the shape of $\calP$ to be the sequence of positive integers $a_1,\ldots,a_n$ defined by \[
a_i = \begin{cases} c_1 & i = 1 \\ c_i - c_{i-1} & 1 < i < n \\ c_n - c_{n-1} + 1 & i = n.
\end{cases}
\]
Notice that $a_1 = 1$ if and only if $\calP(1)$ is maximal and that $a_n > 1$. 
For example, consider the fence poset on the right of Figure \ref{fig:SnakeGraphAndPosetExample}. The numbers $c_i$ are $c_1 = 3, c_2 = 5$ and $c_3 = 6$. Therefore, the shape is $3,2,2$, recovering the sign sequence.

An \emph{order ideal} of a poset $\calP$ is a subset $I \subseteq \calP$ such that $x \in I$ and $y \preceq x$ implies $y \in I$. In other words, an order ideal is ``closed going down''. The set of all order ideals of a poset $\calP$ is denoted $J(\calP)$. The empty set is an order ideal, so here we say an empty poset has one order ideal. Given $x \in \calP$, let $\langle x \rangle$ denote the smallest order ideal containing $x$. An order ideal of the form $\langle x \rangle$ is sometimes called a principal order ideal. 

We have now indexed both snake graphs and posets with finite sequences of positive integers whose last entry is greater than 1. The connection is the following.

\begin{theorem}\label{thm:SnakeGraphAndFencePoset}
Given $a_1,\ldots,a_n$, a list of positive integers with $a_n > 1$, let $\mathcal{G}$ and $\calP$ be the snake graph and poset of shape $a_1,\ldots,a_n$ respectively. There is a bijection between the set of perfect matchings of $\mathcal{G}$ and the set of order ideals of $\calP$.
\end{theorem}

\begin{proof}
This is a corollary of \cite[Theorem 5.4]{musiker2013bases}.
\end{proof}

More recently, a weighted version of Theorem \ref{thm:SnakeGraphAndFencePoset} was given in \cite{ezgieminecluster2024} which allows one to compute cluster variables directly from fence posets. The same result was independently shown in \cite{pilaud2023posets}. Here, we only need the fact that these sets have the same size. Combining Theorems \ref{thm:CountPMSWithContFrac} and \ref{thm:SnakeGraphAndFencePoset} allows us to quickly compute the number of order ideals of a fence poset.

\begin{cor}\label{cor:CountOrderIdeals}
If $\calP$ is a fence poset of shape $a_1,\ldots,a_n$, then the number of order ideals of $\calP$ is $\mathcal{N}[a_1,a_2,\ldots,a_n]$.
\end{cor}

For example, since we previously computed that the snake graph in Figure \ref{fig:SnakeGraphAndPosetExample} has $\mathcal{N}[3,2,2] = 17$ perfect matchings, we also know the fence poset on the righthand side has 17 order ideals.

\begin{remark}
 Corollary \ref{cor:CountOrderIdeals} gives another perspective for  \cite[Lemma 2]{elizalde2021rowmotion}, which gave a recursion on the cardinality of the set of order ideals of a fence poset.
 \end{remark}

It will later be convenient to consider posets equipped with a \emph{weight function}, $w: \calP \to \mathbb{R}$. In this setting,  define \[
\mathcal{W}(\calP) := \sum_{I \in J(\calP)} \prod_{\rho \in I} w(\rho).
\]

As a convention, we extend an unweighted poset $\calP$ to a weighted one by defining $w(\rho) = 1$ for all $\rho \in \calP$. In this case, $\mathcal{W}(\calP)$ recovers the number of order ideals of $\calP$.

\subsection{Skein Relations}\label{sec:Skein}

Let a \emph{multicurve} denote a multiset of curves on a surface.
Given two curves with a designated point of intersection, the \emph{resolution} of this intersection is the pair of multicurves resulting from replacing the point of intersection, locally $\mathsf{X}$, with $\asymp$ and with \rotatebox[origin=c]{90}{$\asymp$}. Understanding skein relations has been important in the study of cluster algebras from surfaces as cluster variables satisfy an algebraic version of skein relations \cite{FT-II}. That is, if $x$ and $x'$ are cluster variables associated to curves which intersect, and the resolution of this intersection is the pair of multicurves $\Gamma^+$ and $\Gamma^-$, each of which corresponds to a product of elements of the cluster algebra, then we have \[
xx' = x_{\Gamma^+} + x_{\Gamma^-}.
\]

Indeed, Vieta jumping for ordinary Markov numbers (Equation \ref{eq:Vieta}) can be seen as a special case of a skein relation in the surface model for the Markov cluster algebra.

There have been multiple approaches to studying skein relations for cluster algebras. One, given by \canakci and Schiffler, uses the snake graph construction and is combinatorial in nature \cite{CS13}. In joint work with Kang and Kelley, we gave poset versions of these results, working in the more general setting of ``tagged arcs'' \cite{banaian2024skein}, which results in considering a slightly larger class than fence posets. We will be interested in the enumerative consequences of these results. Here, we do not need this wider generality and only consider fence posets. Therefore, while we use the language from \cite{banaian2024skein} for convenience, every result discussed also comes from simply translating \cite{CS13} to the land of fence posets. From here on, all posets will be assumed to be fence posets and to be equipped with a weight function.

Let a subset $S$ of elements in a poset $\calP$ be on \emph{bottom} if for any $\rho \in S, \sigma \in \calP \backslash S$, $\rho \not\succeq \sigma$; this is equivalent to the condition of being an order ideal. Define a set to be on \emph{top} in a parallel way. Such sets are sometimes called ``order filters.''

Given two posets $\calP_1$ and $\calP_2$, if there is a weight-preserving isomorphism between two subposets of the form $R_1 = \calP_1[c,d]$ and $R_2 = \calP_2[c',d']$, we say they have an \emph{overlap}. We emphasize that this isomorphism must be consistent with the two chronological orderings.  We say this is moreover a \emph{crossing overlap} if \begin{itemize}
    \item $R_1$ is on top of $\calP_1$ and $R_2$ is on bottom of $\calP_2$; 
    \item we do not have both $c = 1$ and $c' = 1$; and 
    \item we do not have both $d = h_1$ and $d' = h_2$ where $h_i = \vert \calP_i \vert$.
\end{itemize}

\begin{definition}\label{def:Type0Resolve}
Given a pair of posets $\calP_1$ and $\calP_2$ with a crossing overlap $R_1 = \calP_1[c,d] \cong \calP_2[c',d'] = R_2$, define the \emph{Type 0 resolution} of $\calP_1$ and $\calP_2$ to be $\{\calP_3, \calP_4\} \cup \{\calP_5, \calP_6\}$ where these new posets are defined as follows. The associated weight functions are defined in the natural way.
\begin{itemize}
    \item Let $\calP_3$  be the poset on $\calP_1[1,d] \cup \calP_2[d'+1,h_1]$ with all induced relations as well as $\calP_1(d)\prec \calP_2(d'+1)$.
    \item Let $\calP_4$  be the poset on $\calP_2[1,d'] \cup \calP_2[d+1,h_2]$ with all induced relations as well as $\calP_2(d')\succ\calP_2(d+1)$.
    \item The construction of $\calP_5$ depends on $c$ and $c'$.
    \begin{itemize}
        \item If $c>1$ and $c'>1$, then $\calP_5$ is the poset on $\calP_1[1,c-1] \cup \calP_2[1,c'-1]$ with all induced relations as well as $\calP_1(c-1) \succ \calP_2(c'-1)$.
        \item If $c = 1$, implying $c' > 1$, let $v<c'-1$ be the largest integer such that $\calP_2(v) \not\prec \calP_2(c'-1)$, if it exists,  and otherwise let $v = 0$. Let $\calP_5 $  be the induced subposet $\calP_2[1,v]$.
        \item If $c' = 1$, implying $c > 1$, let $u<c-1$ be the largest integer such that $\calP_1(u) \not\succ \calP_1(c-1)$, if it exists,  and otherwise let $u = 0$. Let $\calP_5 $  be the induced subposet $\calP_1[1,u]$. 
    \end{itemize}
    \item There are three parallel cases to construct $\calP_6$, analyzing $d$ and $d'$. See \cite[Definition 8.2]{BG}, where ``$s$'' denotes $c$ here and ``$t$'' denotes $d$ here.
\end{itemize}
\end{definition}

 Definition \ref{def:Type0Resolve} was called a ``Type 0 resolution'' as it corresponds to a certain type of geometric intersection with respect to a triangulation; see Section \ref{sec:InducedOrdering}. There are several other ways in which two curves can intersect. These are only apparent from the posets when the elements are labeled with curves from a triangulation. Any pair of abstract fence posets of given shape can realize such intersections. Therefore, here we discuss them without referencing when they apply.

\begin{definition}\label{def:Type1Resolution}
Given a pair of posets $\calP_1$ and $\calP_2$, choose $1 \leq i < h_2 := \vert \calP_2 \vert$. Chronologically label $\calP_2$ such that $\calP_2(i) \succ \calP_2(i+1)$. Define the \emph{Type 1 resolution} of $\calP_1$ and $\calP_2$ with respect to $i$ to be $\{\calP_3, \calP_4\} \cup \{\calP_5, \calP_6\}$ where these new posets are defined as follows. The associated weight functions are defined in the natural way.
\begin{itemize}
    \item  Let $\calP_3$  be the poset on $\calP_1 \cup \calP_2[1,i]$ with all induced relations as well as $\calP_2(i) \prec \calP_1(1)$.
    \item Let $v>i$ be the smallest integer such that $\calP_2(i) \not\succ \calP_2(v)$, if it exists; otherwise let $v = h_2 + 1$. Let $\calP_4$ be the induced poset on $\calP_2[v,h_2]$.
    \item Let $\calP_5$  be the poset on $\calP_1 \cup \calP_2[i+1,h_2]$ with all induced relations as well as $\calP_2(i+1) \succ \calP_1(1)$.
    \item Let $u<i$ be the largest integer such that $\calP_i(i) \not\prec \calP_2(u)$, if it exists; otherwise let $u=0$. Let  $\calP_6$ be the induced poset on $\calP_2[1,u]$.
\end{itemize}
\end{definition}

\begin{definition}\label{def:Type2Resolution}
Given a pair of posets $\calP_1$ and $\calP_2$, define the \emph{Type 2 resolution} of $\calP_1$ and $\calP_2$ to be $\{\calP_3, \calP_4\} \cup \{\calP_5, \calP_6\}$ where these new posets are defined as follows. The associated weight functions are defined in the natural way. \begin{itemize}
    \item  Let $\calP_3$  be poset on $\calP_1 \cup \calP_2$ with all induced relations as well as $\calP_1(1) \succ \calP_2(1)$.
    \item Let $\calP_4$ be the empty poset. 
    \item Let $v$ be the smallest integer such that $\calP_1(1) \not\succeq  \calP_1(v)$, if it exists; otherwise set $v = \vert \calP_1 \vert=:h_1$. Let $\calP_5$ be the induced poset on $\calP_1[v,h_1]$.
    \item Let $u$ be the smallest integer such that $\calP_2(1) \not\preceq  \calP_2(u)$, if it exists; otherwise set $u = \vert \calP_2 \vert=:h_2$. Let $\calP_6$ be the induced poset on $\calP_2[u,h_2]$.
\end{itemize}
\end{definition}

\begin{prop}\label{prop:SkeinRelationsOnPosets}
Let $\calP_1$ and $\calP_2$ be a pair of posets. Suppose $\calP_3,\calP_4,\calP_5$ and $\calP_6$ are four posets satisfying one of the following conditions.
\begin{enumerate}
    \item The pair $\calP_1$ and $\calP_2$ has a crossing overlap and $\{\calP_3,\calP_4\} \cup \{\calP_5,\calP_6\}$ is the Type 0 resolution.
    \item The set $\{\calP_3,\calP_4\} \cup \{\calP_5,\calP_6\}$ is the Type 1 resolution of $\calP_1$ and $\calP_2$ with respect to some $1 \leq i < \vert \calP_2 \vert$.
    \item The set $\{\calP_3,\calP_4\} \cup \{\calP_5,\calP_6\}$ is the Type 2 resolution of $\calP_1$ and $\calP_2$. 
\end{enumerate}
Then, there exists a monomial $W$ such that  \[
\mathcal{W}(\calP_1) \mathcal{W}(\calP_2)  = \mathcal{W}(\calP_3) \mathcal{W}(\calP_4)  + W \cdot \mathcal{W}(\calP_5) \mathcal{W}(\calP_6) .
\]
\end{prop}

\begin{proof}
This is the consequence of \cite[Propositions 6,7,8]{banaian2024skein}, or equivalently from translating the results from \cite{CS13} to statements concerning posets.
\end{proof}

As we will be focusing on positive weights and largely using Proposition \ref{prop:SkeinRelationsOnPosets} for inequalities, the precise formula for $W$ is not important here. 

\begin{example}\label{ex:Type0Resolution}
First, consider the two posets $\calP_1$ and $\calP_2$ below. We denote elements in each with their chronological label, using primes for $\calP_2$. These two posets have a crossing overlap in $\calP_1[2,4] \cong \calP_2[1',3']$.

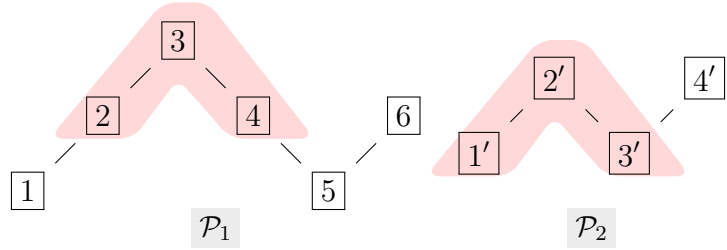
\begin{center}
\begin{tikzpicture}
\filldraw [rounded corners=5pt, red!15]
(10.3,1.2)  
-- (11.8,3) 
-- (12.4,3) 
-- (13.8,1.2)  
-- (12.7,1.2) 
-- (12,2)
-- (11.4,1.2)
-- cycle;
\filldraw [rounded corners=5pt, red!15]
(15.3,0.7)  
-- (16.8,2.5) 
-- (17.4,2.5) 
-- (18.8,.7)  
-- (17.7,.7) 
-- (17,1.5)
-- (16.4,.7)
-- cycle;
 \node(1) at (10,0.5){$\boxed{1}$};
\node(2) at (11,1.5){$\boxed{2}$};
\node(3) at (12,2.5){$\boxed{3}$};
\node(4) at (13,1.5){$\boxed{4}$};
\node(5) at (14,0.5){$\boxed{5}$};
\node(6) at (15,1.5){$\boxed{6}$};
\draw(1) -- (2);
\draw(2) -- (3);
\draw(3) -- (4);
\draw(4) -- (5);
\draw(5) -- (6); 
%%%%%%%%%%%%%%%%%%%
\node[] at (12.5,0){\highlight{$\calP_1$}};
\node(1') at (16,1){$\boxed{1'}$};
\node(2') at (17,2){$\boxed{2'}$};
\node(3') at (18,1){$\boxed{3'}$};
\node(4') at (19,2){$\boxed{4'}$};
\draw(1') -- (2');
\draw(2') -- (3');
\draw(3') -- (4');
\node[] at (17.5,0){\highlight{$\calP_2$}};
\end{tikzpicture}
\end{center}

The resolution consists of the following posets. We do not draw $\calP_5$ as it is empty. 

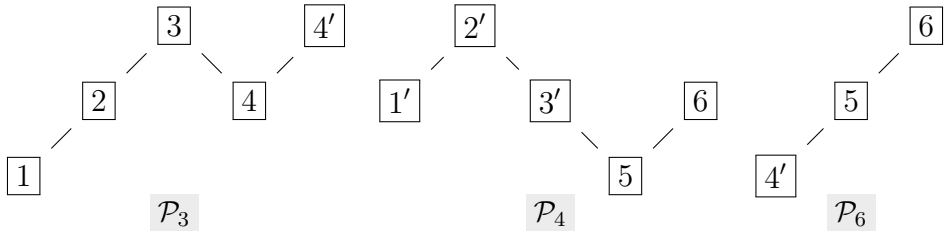
\begin{center}
\begin{tikzpicture}
 \node(1) at (10,0){$\boxed{1}$};
\node(2) at (11,1){$\boxed{2}$};
\node(3) at (12,2){$\boxed{3}$};
\node(4) at (13,1){$\boxed{4}$};
\node(5) at (14,2){$\boxed{4'}$};
\draw(1) -- (2);
\draw(2) -- (3);
\draw(3) -- (4);
\draw(4) -- (5);
\node[] at (12,-0.5){\highlight{$\calP_3$}};
\node(1') at (15,1){$\boxed{1'}$};
\node(2') at (16,2){$\boxed{2'}$};
\node(3') at (17,1){$\boxed{3'}$};
\node(4') at (18,0){$\boxed{5}$};
\node(5') at (19,1){$\boxed{6}$};
\draw(1') -- (2');
\draw(2') -- (3');
\draw(3') -- (4');
\draw(4') -- (5');
\node[] at (17,-0.5){\highlight{$\calP_4$}};
\node(10) at (20,0) {$\boxed{4'}$};
\node(11) at (21,1){$\boxed{5}$};
\node(12) at (22,2){$\boxed{6}$};
\draw(10) -- (11);
\draw(11) -- (12);
\node[] at (21,-0.5){\highlight{$\calP_6$}};
\end{tikzpicture}
\end{center}

First set the weight of each element to 1. One can check the existence of a bijection between the sets of order ideals using Corollary \ref{cor:CountOrderIdeals}. As previously noted, $\calP_1$ has 17 order ideals. One can further compute $\calP_2$ has $\mathcal{N}[2,1,2] = 8$ order ideals and similarly $\calP_3$ has 11, $\calP_4$ has 12, $\calP_5$ has 1, and $\calP_6$ has 4. Indeed, $17 \cdot 8 = 136 =  11 \cdot 12 + 1 \cdot 4$. If we consider a more general weight polynomial with $w_i = w(\calP_1(i))$ and $w_i' = w(\calP_2(i))$, then we can calculate $\mathcal{W}(\calP_1) \mathcal{W}(\calP_2) = \mathcal{W}(\calP_3) \mathcal{W}(\calP_4) + w_1'w_2'w_3' \mathcal{W}(\calP_5) \mathcal{W}(\calP_6)$.
\end{example}

\begin{example}\label{ex:Type2Resolution}
We demonstrate the application of  Definition \ref{def:Type2Resolution} to $\calP_1$ and $\calP_2$ from Example \ref{ex:Type0Resolution}. First, we draw $\calP_3$.

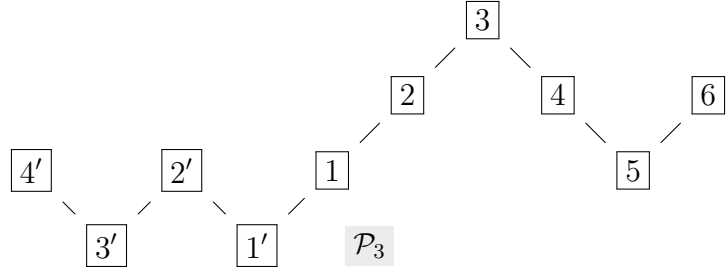
\begin{center}
\begin{tikzpicture}
\node(4') at (6,0.5){$\boxed{4'}$};
\node(3') at (7,-0.5){$\boxed{3'}$};
\node(2') at (8,0.5){$\boxed{2'}$};
\node(1') at (9,-0.5){$\boxed{1'}$};
 \node(1) at (10,0.5){$\boxed{1}$};
\node(2) at (11,1.5){$\boxed{2}$};
\node(3) at (12,2.5){$\boxed{3}$};
\node(4) at (13,1.5){$\boxed{4}$};
\node(5) at (14,0.5){$\boxed{5}$};
\node(6) at (15,1.5){$\boxed{6}$};
\draw(1) -- (2);
\draw(2) -- (3);
\draw(3) -- (4);
\draw(4) -- (5);
\draw(5) -- (6); 
\draw(1') -- (1);
\draw(1') -- (2');
\draw(2') -- (3');
\draw(3') -- (4');
\node[] at (10.5,-0.5){\highlight{$\calP_3$}};
\end{tikzpicture}
\end{center}

Now, the number $v$ from Definition \ref{def:Type2Resolution} is 2 since $\calP_1(1) \not \succeq  \calP_1(2)$. The number $u$ is $3$ since $\calP_2(1) \preceq \calP_2(2)$ but $\calP_2(1) \not\preceq  \calP_2(3)$. Therefore, the posets $\calP_5$ and $\calP_6$ are as below.

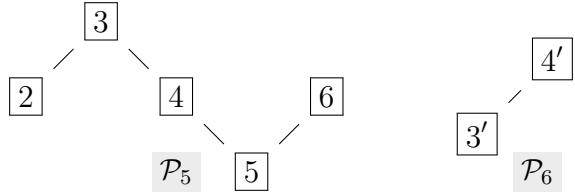
\begin{center}
\begin{tikzpicture}
\node(2) at (1,1.5){$\boxed{2}$};
\node(3) at (2,2.5){$\boxed{3}$};
\node(4) at (3,1.5){$\boxed{4}$};
\node(5) at (4,0.5){$\boxed{5}$};
\node(6) at (5,1.5){$\boxed{6}$};
\draw(2) -- (3);
\draw(3) -- (4);
\draw(4) -- (5);
\draw(5) -- (6); 
\node[] at (3,0.5){\highlight{$\calP_5$}};
\node(3') at (7,1){$\boxed{3'}$};
\node(4') at (8,2){$\boxed{4'}$};
\draw(3') -- (4');
\node[] at (7.8,0.5){\highlight{$\calP_6$}};
\end{tikzpicture}
\end{center}

Using continued fractions, we can numerically verify Proposition \ref{prop:SkeinRelationsOnPosets} when all weights are 1. In Example \ref{ex:Type0Resolution}, we saw $\vert J(\calP_1) \vert \cdot \vert J(\calP_2 ) \vert = 17 \cdot 8 = 136$. The shape of $\calP_3$ is $1,1,1,1,3,2,2$, the shape of $\calP_5$ is $2,2,2$, and the shape of $\calP_6$ is 3. Indeed, we have \[
\mathcal{N}[1,1,1,1,3,2,2] + \mathcal{N}[2,2,2] \cdot \mathcal{N}[3] = 100 + 12 \cdot 3 = 136.
\]
If we consider a more general weight polynomial with $w_i = w(\calP_1(i))$ and $w_i' = w(\calP_2(i))$, then we can calculate $\mathcal{W}(\calP_1) \mathcal{W}(\calP_2) = \mathcal{W}(\calP_3) \mathcal{W}(\calP_4) + w_1 \mathcal{W}(\calP_5) \mathcal{W}(\calP_6)$.
\end{example}

From the point of view of Corollary \ref{cor:CountOrderIdeals}, Proposition \ref{prop:SkeinRelationsOnPosets} can alternately be viewed as a family of relations on numerators of continued fractions. There are many such formulas, based on the various local configurations of these poset families. We present two of these formulas.

\begin{cor}
Let $a_1,\ldots,a_n$ and $b_1,\ldots,b_m$ be two finite sequences of positive integers.  Let $\alpha$ denote the sequence $a_1,\ldots,a_n$, $\beta$ denote the sequence $b_1,\ldots,b_m$, and $\overline{\alpha}$ denote the reverse sequence $a_n,\ldots,a_1$.

\begin{enumerate}
    \item Let $1 < i < m$ and let $b',b'' \in \mathbb{Z}_{\geq 0}$ satisfy $b' + b'' + i = b_i$. Let $\beta'$ denote the sequence $b_1,\ldots,b_{i-2}$ and $\beta''$ the sequence $b_{i+2},\ldots,b_m$. We have \[
\mathcal{N}[\alpha] \mathcal{N}[\beta] = \mathcal{N}[\beta',b_{i-1},b',\alpha]\mathcal{N}[b_{i+1}+1,\beta''] + \mathcal{N}[\overline{\alpha},1,b'',b_{i+1},\beta''] \mathcal{N}[\beta',b_{i-1}+1]
    \]
    \item Suppose $a_1 = 1$ and $b_1 > 1$. We have \[
\mathcal{N}[\alpha]\mathcal{N}[\beta] = \mathcal{N}[\overline{\alpha},\beta] + \mathcal{N}[a_2,\ldots,a_n]\mathcal{N}[b_2,\ldots,b_n]
    \]
\end{enumerate}
\end{cor}

\begin{proof}
Part 1 comes from applying Proposition \ref{prop:SkeinRelationsOnPosets} for a resolution of Type 1 when the element $\calP_2(i)$ is incomparable with $\calP_2(1)$ and $\calP_2(\vert \calP_2 \vert)$, i.e., this element does not lie on the first or last maximal chain of $\calP_2$. Part 2 comes from applying the same Proposition for a resolution of Type 2 where $\calP_1(1)$ is maximal and $\calP_2(1)$ is minimal.
\end{proof}

We will also consider a construction related to a single poset which is inspired by certain self-intersecting curves on an orbifold.  
 We say a poset $\calP$ has a \emph{reverse self-overlap} if it has two subposets $\calP[c_1,d_1]$ and $\calP[c_2,d_2]$ which are related by a isomorphism $\Psi$ which acts by $\Psi(c_1 + i) = \Psi(d_2-i)$ for all $0 \leq i \leq d_1 - c_1 = d_2 - c_2$ and such that the weight of $\calP(c_1+i)$ is equal to the weight of $\calP(d_2-i)$ for all $0 \leq i < d_1-c_1$. We moreover say that a reverse self-overlap is \emph{crossing} if $\calP[c_1,d_1]$ is on top, $\calP[c_2,d_2]$ is on bottom, and we do not have both $c_1 = 1$ and $d_2 = \vert \calP \vert$.  Finally, we say that a reverse crossing self-overlap is \emph{kissing} if $c_2 = d_1 + 1$. In particular, if $\calP$ has a  reverse crossing self-overlap, then $\calP(d_1) \succ \calP(c_2)$.
 
In joint work with Gyoda, we considered the poset skein relation coming from reverse kissing self-overlaps. 

\begin{definition}\label{def:ResolveRKSO}\cite[Definition 8.9]{BG} 
Given a weighted poset $\calP$ with a reverse kissing self-overlap in $\calP[c_1,d_1]$ and $\calP[c_2,d_2]$, define the \emph{resolution} of $\calP$ to be $\{\calP_{34}\} \cup \{\calP_{56}\}$ where these new posets are defined as follows. 
\begin{itemize}
    \item Let $\calP_{34}$ be the poset on $\calP$ where we have all the same relations except $\calP(d_1) \prec \calP(c_2)$. Define the weight function on $\calP_{34}$ to be $w(\calP_{34}(i)) = w(\calP(i))$ if $i \neq d_1,c_2$, $w(\calP_{34}(d_1)) = w(\calP(c_2))$ and $w(\calP_{34}(c_2)) = w(\calP(d_1))$.
    \item The definition of $\calP_{56}$ depends on $c_1$ and $d_2$. In each case, the weight function is defined in the natural way. Let $h:= \vert \calP \vert$.
    \begin{itemize}
        \item If $c_1>1$ and $d_2 < \vert \calP \vert = h$, let $\calP_{56}$ be the poset on $\calP[1,c_1-1] \cup \calP[d_2+1,h]$ with all induced relations as well as $\calP_{56}(c_1-1) \succ \calP_{56}(d_2+1)$.
        \item If $c_1 = 1$, let $v > d_2+1$ be the smallest integer such that $\calP(v) \not\prec \calP(d_2+1)$, it is exists, and otherwise set $v = h+1$. Let $\calP_{56} = \calP[v,h]$.
        \item If $d_2 = h$, we define $\calP_{56}$ in a parallel way. 
    \end{itemize} 
\end{itemize}
\end{definition}

\begin{prop}\cite[Proposition 8.12]{BG}\label{prop:RKO}
If $\calP$ is a poset with a reverse kissing self-overlap in $\calP[c_1,d_1]$ and $\calP[c_2,d_2]$, then there is a monomial $W$ such that \[
\mathcal{W}(\calP) = \mathcal{W}(\calP_{34}) + W\cdot\mathcal{W}(\calP_{56}).
\]
\end{prop}

\begin{example}
Consider the poset $\calP_1$ from Example \ref{ex:Type0Resolution} as a weighted poset with all elements weighted 1. Then, $\calP_1$ has a self-crossing overlap in $\calP[2,3]$ and $\calP[4,5]$. We reproduce this poset below as well as the posets $\calP_{34}$ and $\calP_{56}$ from Definition \ref{def:ResolveRKSO}.

\begin{center}
\begin{tikzpicture}
\filldraw [rounded corners=5pt, red!15]
(10.3,1.2)  
-- (11.8,3) 
-- (12.4,3) 
-- (12.4,2.3)
-- (11.4,1.2)
-- cycle;
\filldraw [rounded corners=5pt, red!15]
(12.4,1.8)  
-- (12.8,2.3) 
-- (14.6,0.5) 
-- (13.7,-0.1)
-- cycle;
 \node(1) at (10,0.5){$\boxed{1}$};
\node(2) at (11,1.5){$\boxed{2}$};
\node(3) at (12,2.5){$\boxed{3}$};
\node(4) at (13,1.5){$\boxed{4}$};
\node(5) at (14,0.5){$\boxed{5}$};
\node(6) at (15,1.5){$\boxed{6}$};
\draw(1) -- (2);
\draw(2) -- (3);
\draw(3) -- (4);
\draw(4) -- (5);
\draw(5) -- (6); 
\node[] at (12.5,0){\highlight{$\calP_1$}};
%%%%%%%%%%%%%%%%%%%
 \node(1) at (16,0.5){$\boxed{1}$};
\node(2) at (17,1.5){$\boxed{2}$};
\node(3) at (18,2.5){$\boxed{3}$};
\node(4) at (19,3.5){$\boxed{4}$};
\node(5) at (20,2.5){$\boxed{5}$};
\node(6) at (21,3.5){$\boxed{6}$};
\draw(1) -- (2);
\draw(2) -- (3);
\draw(3) -- (4);
\draw(4) -- (5);
\draw(5) -- (6); 
\node[] at (18.5,0){\highlight{$\calP_{34}$}};
%%%%%%%%%%%%%%%%%%%%
\node(1) at (23,2.5){$\boxed{1}$};
\node(6) at (24,1.5){$\boxed{6}$};
\draw(1) -- (6);
\node[] at (23.5,0){\highlight{$\calP_{56}$}};
\end{tikzpicture}
\end{center}

As previously noted, $\calP_1$ has 17 order ideals. The poset $\calP_{34}$ has $\mathcal{N}[4,1,2] = 14$ order ideals and $\calP_{56}$ clearly has 3 order ideals. 
If we consider a more general weight polynomial with $w_i = w(\calP_1(i))$, then one can compute  $\mathcal{W}(\calP_1) = \mathcal{W}(\calP_{34}) + w_4 w_5 \mathcal{W}(\calP_{56})$.
\end{example}

\section{Two Poset Constructions}\label{sec:TwoPosets}

In \cite{BG} we constructed a labeled, weighted fence poset $\calP_\gamma$ associated to any curve $\gamma$ on the lattice $\mathbb{Z}^2$ with integral endpoints which does not contain integral points in its interior. Refer to such a curve as an \emph{arc}; we will also orient each arc in one of two directions. Whereas therein elements were labeled with Laurent monomials, here we will specialize all formal variables to 1 and have rational weights. For book-keeping, we will sometimes label the elements as well. When the weights can vary, we will refer to poset elements  by $(l,w)$ where $l$ is the label and $w$ is the weight. Soon, we will focus on posets with all elements weighted with 1, and then we will just refer to labels. We emphasize that this is distinct from the chronological labeling discussed in the previous section and in particular multiple elements of a poset can have the same label.

Let $\mathcal{L}$ be the lattice of all arcs (or, ``line segments'') in $\mathbb{R}^2$ of slope 0, $\infty$, and $-1$ with endpoints being integral.  
 We will assume all arcs have a minimal set of intersections with $\mathcal{L}$. In particular, an arc will never cross the same arc  in $\mathcal{L}$ two consecutive times. We will also always choose one of two possible orientations for an arc. This induces an indexing on the arc's intersections with $\mathcal{L}$.

\begin{algorithm}\label{algo:Old}
Let $\gamma$ be an arc on $\mathbb{Z}^2$.  If $\gamma$ does not intersect any line segments in $\mathcal{L}$, then $\calP_\gamma$ is empty.

Otherwise, suppose the first line segment crossed by $\gamma$ is $\tau_a$. If $k = 0$, then introduce one new element $\calP_\gamma(1) = (\tau_a,1)$ to $\calP_\gamma$. Otherwise, introduce two new elements $\calP_\gamma(1)$ and $\calP_\gamma(2)$. If the intersection point of $\gamma$ and $\tau_a$ lies strictly closer to the right endpoint of $\tau_a$ than the left (with respect to the orientation of $\gamma)$, then set $\calP_\gamma(1) \succ \calP_\gamma(2)$; otherwise set $\calP_\gamma(1) \prec \calP_\gamma(2)$. Give both elements the label $\tau_a$, weight the smaller with $k$ and weight larger with $\frac{1}{k}$.

The remainder of the poset is constructed by performing the following.
\begin{enumerate}
    \item Let $\tau_a$ be the last arc from $\mathcal{L}$ which we have accounted for. If $\tau_a$ is also the last arc crossed by $\gamma$, then we return the poset constructed. Otherwise, let $\tau_b$ be the next arc crossed by $\gamma$, and let $\calP_\gamma(j)$ be the last element of $\calP_\gamma$ constructed so far.
    \item Introduce an element $\calP_\gamma(j+1)$ which has label $\tau_b$. If $\tau_a$ and $\tau_b$ share an endpoint to the right of $\gamma$, set $\calP_\gamma(j) \succ \calP_\gamma(j+1)$; otherwise,  set $\calP_\gamma(j) \prec \calP_\gamma(j+1)$.
    \item If $k = 0$, then weight $\calP_\gamma(j+1)$ with 1 and return to step 1. 
    \item If $k > 0$, then introduce another element $\calP_\gamma(j+2)$ which is also labeled $\tau_b$. If the intersection point of $\gamma$ and $\tau_b$ lies strictly closer to the right endpoint of $\tau_b$ than the left (with respect to the orientation of $\gamma)$, then set $\calP_\gamma(j+1) \succ \calP_\gamma(j+2)$; otherwise set $\calP_\gamma(j+1) \prec \calP_\gamma(j+2)$. In either case, weight the smaller element with $k$ and weight larger with $\frac{1}{k}$. Return to Step (1).
\end{enumerate}  
\end{algorithm}

Notice that even if $\gamma$ and $\delta$ are homotopic arcs and have the same sequence of intersections with $\mathcal{L}$, the posets $\calP_\gamma$ and $\calP_\delta$ can be different.  This is illustrated in the following Example.

\begin{example}\label{Ex:DoesNotDistinguishHomotopy}
Consider the following two arcs,  $\gamma_1$ and $\gamma_2$.  

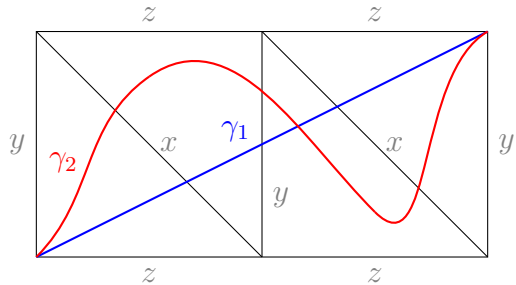
\begin{center}
\begin{tikzpicture}[scale=3]
\draw (0,0) to  (1,0) to  (2,0) to (2,1) to node[above, gray]{$z$}(1,1) to(0,1) to (0,0);
\draw (1,0) to node[right,yshift = -20pt, gray]{$\tau_2$} (1,1);
\draw (0,1) to node[right,gray]{$\tau_1$} (1,0);
\draw(1,1) to node[right,gray]{$\tau_3$} (2,0);
\draw[thick, blue] (0,0) to node[above, xshift = -10pt, yshift = -3pt]{$\gamma_1$} (2,1);
\draw[thick, red] (0,0) to [out = 45, in = 215] node[left]{$\gamma_2$} (0.5,0.8);
\draw[thick, red] (0.5,0.8) to [out = 35, in = 135] (1.5,0.2);
\draw[thick, red] (1.5,0.2) to [out = -45,  in = 210] (2,1);
\end{tikzpicture}
\end{center}

Assume first that $k>0$.  Then, the two posets, $\calP_{\gamma_1}$ and $\calP_{\gamma_2}$ are drawn on the left and right below, respectively. Each element is denoted with as (label, weight).

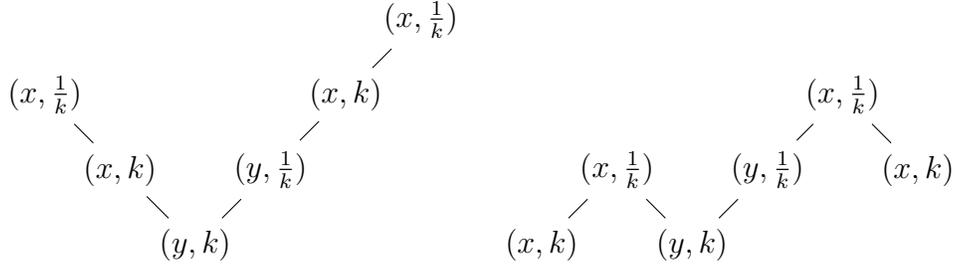
\begin{center}
\begin{tabular}{cc}
\begin{tikzpicture}
\node(1) at (0,0){$(\tau_1,\frac{1}{k})$};
\node(2) at (1,-1){$(\tau_1,k)$};
\node(3) at (2,-2){$(\tau_2,k)$};
\node(4) at (3,-1){$(\tau_2,\frac{1}{k})$};
\node(5) at (4,0){$(\tau_3,k)$};
\node(6) at (5,1){$(\tau_3,\frac{1}{k})$};
\draw(1) -- (2);
\draw(2) -- (3);
\draw(3) -- (4);
\draw(4) -- (5);
\draw(5) -- (6);
\end{tikzpicture}&
\begin{tikzpicture}
\node(1) at (0,0){$(\tau_1,k)$};
\node(2) at (1,1){$(\tau_1,\frac{1}{k})$};
\node(3) at (2,0){$(\tau_2,k)$};
\node(4) at (3,1){$(\tau_2,\frac{1}{k})$};
\node(5) at (4,2){$(\tau_3,\frac{1}{k})$};
\node(6) at (5,1){$(\tau_3,k)$};
\draw(1) -- (2);
\draw(2) -- (3);
\draw(3) -- (4);
\draw(4) -- (5);
\draw(5) -- (6);
\end{tikzpicture}
\end{tabular}
\end{center}

If $k = 0$, the resulting posets are in fact the same.  They are given by $(\tau_1,1) \succ (\tau_2,1) \prec (\tau_3,1)$. 
\end{example}

Let $\gamma_\fpq$ denote the a straight line between $(0,0)$ and $(q,p)$ where $p$ and $q$ are coprime positive integers. For instance, $\gamma_1$ in Example \ref{Ex:DoesNotDistinguishHomotopy} is $\gamma_{\frac12}$. Each such arc intersects one line segment at its midpoint. Note that our algorithm is \emph{left-biased} in that even when $\gamma$ crosses a line segment $\tau$ at its midpoint, we treat the intersection as if it is closer to the left endpoint.  A corollary from the work in \cite{BG}  is that the posets $\calP_{\fpq}$ can be used to compute the $k$-Markov numbers.

\begin{theorem}\label{thm:ComputeKMarkov}
If $0 < p < q$ satisfy $\gcd(p,q) = 1$ and $k \geq 0$, then \[
\mathcal{W}(\calP_{\gamma_\fpq}) = m^{(k)}_{\fpq}.
\]
\end{theorem}

\begin{proof}
This is a consequence of setting all variables $x_i$ and $y_i$ to 1 in \cite[Theorem 8.30]{BG}.
\end{proof}

The statement of Theorem \ref{thm:ComputeKMarkov} points out that the parameter $k$ is implicit in our notation of the poset $\calP_\fpq$. Since we will not vary $k$, this should not cause confusion.

We will be interested in a wider family of curves. Let $A,B \in \mathbb{Z}^2$. Define $\gamma_{AB}^L$ to be the result of taking the line segment $\overline{AB}$ and deforming each intersection of $\overline{AB}$ with  points in $\mathbb{Z}^2$ other than $A$ and $B$ an infinitesimal amount to the left.  
%Here, ``to the left'' is with respect to orienting the path from $A$ to $B$.  The first criteria means that $\gamma_{AB}^L$ will not intersect any line segments at exactly their midpoint. The second criteria implies that $A$ and $B$ are the only points in $\mathbb{Z}^2$ which $\gamma_{AB}^L$ intersects. 
Define $\gamma_{AB}^R$ similarly to $\gamma_{AB}^L$, using a right deformation instead. 

\begin{convention}\label{conv:gammaABLR}
The arcs $\gamma_{AB}^L$ and $\gamma_{AB}^R$ each have the form of a series of arcs $\gamma_\fpq$ for the same fraction $\fpq$ connected by semicircles. While the arcs $\gamma_\fpq$ each intersect a line segment in $\mathcal{L}$ at its midpoint, we imagine that in $\gamma_{AB}^L$, these intersections have moved slightly to the left and in $\gamma_{AB}^R$ these have moved slightly to the right. 
\end{convention}

In particular, $\calP_{\gamma_\fpq} = \calP_{\gamma^L_{(0,0)(q,p)}}$, and this is distinct from $\calP_{\gamma^R_{(0,0)(q,p)}}$ unless $k = 0$.

\begin{example}
Here, we consider $\gamma_{AB}^L$ and $\gamma_{AB}^R$ where $A = (0,0)$ and $B = (2,4)$.  These are drawn on the left and right below.
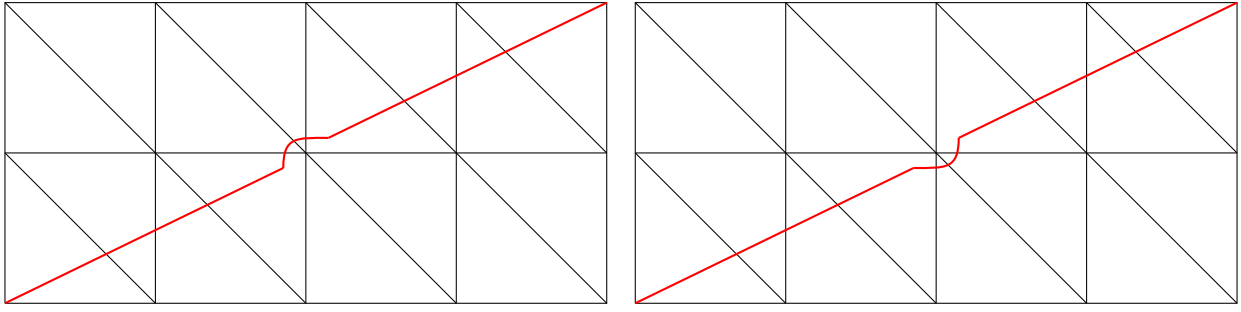
\begin{center}
\begin{tabular}{cc}
\begin{tikzpicture}[scale=2]
\draw (0,0) -- (4,0) -- (4,2) -- (0,2) -- (0,0);
\draw (1,0) -- (1,2);
\draw (2,0) -- (2,2);
\draw(3,0) -- (3,2);
\draw(0,1) -- (4,1);
\draw (0,1) -- (1,0);
\draw (0,2) -- (2,0);
\draw(1,2) -- (3,0);
\draw(2,2) -- (4,0);
\draw(3,2) -- (4,1);
\draw[thick, red] (0,0) -- (1.85,0.9);
\draw[thick, red] (1.85,0.9) to [out = 90, in = 180,looseness=1.5] (2.15,1.1);
\draw[thick,red] (2.15,1.1) -- (4,2);
\end{tikzpicture}&
\begin{tikzpicture}[scale=2]
\draw (0,0) -- (4,0) -- (4,2) -- (0,2) -- (0,0);
\draw (1,0) -- (1,2);
\draw (2,0) -- (2,2);
\draw(3,0) -- (3,2);
\draw(0,1) -- (4,1);
\draw (0,1) -- (1,0);
\draw (0,2) -- (2,0);
\draw(1,2) -- (3,0);
\draw(2,2) -- (4,0);
\draw(3,2) -- (4,1);
\draw[thick, red] (0,0) -- (1.85,0.9);
\draw[thick, red] (1.85,0.9) to [out = 0, in = -90,looseness=1.5] (2.15,1.1);
\draw[thick,red] (2.15,1.1) -- (4,2);
\end{tikzpicture}
\end{tabular}
\end{center}

The poset $\calP_{\gamma_{AB}^L}$ is drawn below.

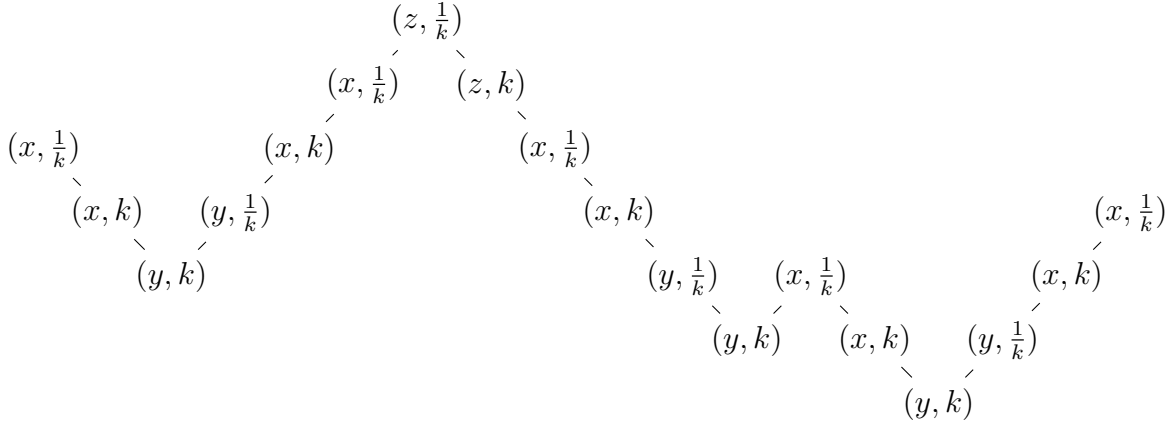
\begin{center}
\begin{tikzpicture}[scale=0.85]
\node(1) at (0,0){$(\tau_1,\frac{1}{k})$};
\node(2) at (1,-1){$(\tau_1,k)$};
\node(3) at (2,-2){$(\tau_2,k)$};
\node(4) at (3,-1){$(\tau_2,\frac{1}{k})$};
\node(5) at (4,0){$(\tau_3,k)$};
\node(6) at (5,1){$(\tau_3,\frac{1}{k})$};
\draw(1) -- (2);
\draw(2) -- (3);
\draw(3) -- (4);
\draw(4) -- (5);
\draw(5) -- (6);
\node(7) at (6,2){$(\tau_4,\frac{1}{k})$};
\node(8) at (7,1){$(\tau_4,k)$};
\node(9) at (8,0){$(\tau_5,\frac{1}{k})$};
\node(10) at (9,-1){$(\tau_5,k)$};
\node(b) at (10,-2){$(\tau_6,\frac{1}{k})$};
\node(c) at (11,-3){$(\tau_6,k)$};
\draw(6) -- (7);
\draw(7) -- (8);
\draw(8) -- (9);
\draw(9) -- (10);
\draw(10) -- (b);
\draw(b) -- (c);
\node(11) at (12,-2){$(\tau_7,\frac{1}{k})$};
\node(12) at (13,-3){$(\tau_7,k)$};
\node(13) at (14,-4){$(\tau_8,k)$};
\node(14) at (15,-3){$(\tau_8,\frac{1}{k})$};
\node(15) at (16,-2){$(\tau_9,k)$};
\node(16) at (17,-1){$(\tau_9,\frac{1}{k})$};
\draw(c)--(11);
\draw(11) -- (12);
\draw(12) -- (13);
\draw(13) -- (14);
\draw(14) -- (15);
\draw(15) -- (16);
\end{tikzpicture}
\end{center}

Notice that $\calP_{\gamma_{AB}^L}$  is the result of taking two copies of $\calP_{\gamma_{\frac12}}$ ($\calP_{\gamma_1}$ in Example \ref{Ex:DoesNotDistinguishHomotopy}) and connecting each to a chain $(\tau,\frac{1}{k}) \succ \cdots  (\tau',k)$. The poset $\calP_{\gamma_{AB}^R}$ can be constructed the same way, but with the order of the chain reversed,  and the middle inequalities of the two copies of $\calP_{\gamma_{\frac12}}$ reversed (see Convention \ref{conv:gammaABLR}).  Indeed, these two posets are the same with reverse chronological labeling.  We will show this is always the case in Lemma \ref{lem:LeftAndRightEqual}.
\end{example}

Our goal is to show that, in this setting, we can replace the posets constructed in \cite{BG} with different posets, called \emph{extended posets}, in which every element has weight 1. Therefore, we can recognize all $k$-Markov numbers as enumerating order ideals and as numerators of continued fractions.

In order to work in slightly larger generality, say that a weighted poset with positive rational weights and chronological ordering is \emph{balanced} if all elements with a weight other than 1 come in pairs of the form $\{\calP(i),\calP(i+1)\}$ such that $\wt(\calP(i))\wt(\calP(i+1)) = 1$ and the smaller of the two elements has an integral weight. For any curve $\gamma$ on $\mathbb{Z}^2$, the poset $\calP_\gamma$ is balanced.

\begin{algorithm}\label{algo:New}
Let $\calP$ be a balanced poset. We construct the \emph{extended poset}, $\calP^\exx$ by performing the following. If every element of $\calP$ has weight 1, then $\calP^\exx = \calP$. Otherwise, let $i$ be the smallest number such that $\calP(i)$ has a weight other than 1. We know that either (1) $\wt(\calP(i)) = k$, $\wt(\calP(i+1)) = \frac1k$, and $\calP(i) \prec \calP(i+1)$, or (2) $\wt(\calP(i)) = \frac1k$, $\wt(\calP(i+1)) = k$, and $\calP(i) \succ \calP(i+1)$. 
    \begin{enumerate}
    \item If we are in case (1), introduce elements $\calP(i+j)$ for $2 \leq j \leq k$ such that $\calP(i+j) \prec \calP(i+j+1)$ for all $1 \leq j < k$. Shift the chronological ordering on all later elements. If $\calP(a)$ is such that $\calP(a) \prec \calP(i+1)$, we remove this relation and replace it with $\calP(a) \prec \calP(i+k)$. If $\calP(a)$ is such that $\calP(a) \succ \calP(i+1)$, we also set $\calP(a) \succ \calP(i+k)$. For all $0 \leq j \leq k$, we set the  weight of $\calP(i+j)$ to $1$. Now, return to the beginning.
    \item If we are in case (2), we perform a dual construction to the previous. 
    \end{enumerate}

\end{algorithm}

When we apply this Algorithm to $\calP_\gamma$, the elements in each ``balanced'' pair will be labeled in the same way, so we can unambiguously label the extended version with this common label. 
Figure \ref{fig:AlgoNew} succinctly illustrates Algorithm \ref{algo:New} along with this labeling convention. 

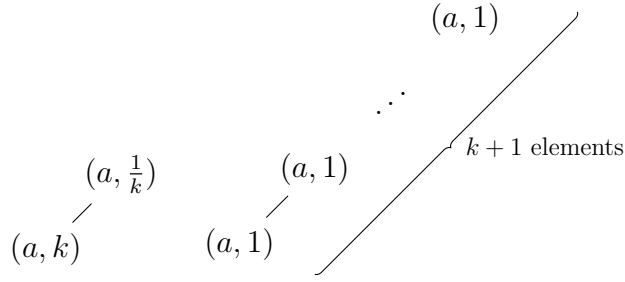
\begin{figure}
\centering
\begin{tabular}{cc}
\begin{tikzpicture}
\node(3) at (2,-2){$(\tau,k)$};
\node(4) at (3,-1){$(\tau,\frac{1}{k})$};
\draw(3) -- (4);
\end{tikzpicture}&
\begin{tikzpicture}
\node(1) at (0,0){$(\tau,1)$};
\node(2) at (1,1){$(\tau,1)$};
\node(3) at (2,2){$\iddots$};
\node(4) at (3,3){$(\tau,1)$};
\draw(1) -- (2);
\draw [decorate,
    decoration = {brace,mirror}] (1,-0.4) to node[right, scale = 0.8, xshift = 5pt, yshift = -2pt]{$k+1$ elements}  (4.5,3.5+-0.4);
\end{tikzpicture}
\end{tabular}
\caption{The main process in Algorithm \ref{algo:New}}\label{fig:AlgoNew}
\end{figure}

%Given a balanced poset $\calP$, the associated extended poset $\calP^\exx$ is also balanced.

\begin{example}\label{ex:ExtendedPoset}
Suppose $k = 2$.  Then,  the extended posets $\calP^\exx_{\gamma_1}$ and $\calP^\exx_{\gamma_2}$ associated to posets $\calP_{\gamma_1}$  and $\calP_{\gamma_2}$ built in Example \ref{Ex:DoesNotDistinguishHomotopy} are below, on the left and right respectively.  Since the weight of each element is one, we suppress weight and only write the label of each element.

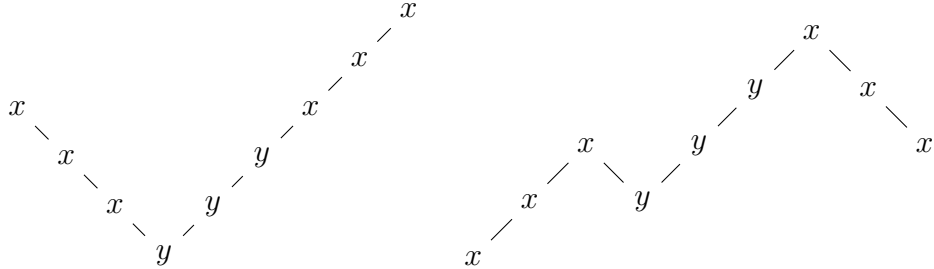
\begin{center}
\begin{tabular}{cc}
\begin{tikzpicture}[scale=0.65]
\node(1) at (0,0){$\tau_1$};
\node(2) at (1,-1){$\tau_1$};
\node(3) at (2,-2){$\tau_1$};
\node(4) at (3,-3){$\tau_2$};
\node(5) at (4,-2){$\tau_2$};
\node(6) at (5,-1){$\tau_2$};
\node(7) at (6,0){$\tau_3$};
\node(8) at (7,1){$\tau_3$};
\node(9) at (8,2){$\tau_3$};
\draw(1) -- (2);
\draw(2) -- (3);
\draw(3) -- (4);
\draw(4) -- (5);
\draw(5) -- (6);
\draw(6) -- (7);
\draw(7) -- (8);
\draw(9) -- (8);
\end{tikzpicture}&
\begin{tikzpicture}[scale=0.75]
\node(1) at (0,0){$\tau_1$};
\node(2) at (1,1){$\tau_1$};
\node(3) at (2,2){$\tau_1$};
\node(4) at (3,1){$\tau_2$};
\node(5) at (4,2){$\tau_2$};
\node(6) at (5,3){$\tau_2$};
\node(7) at (6,4){$\tau_3$};
\node(8) at (7,3){$\tau_3$};
\node(9) at (8,2){$\tau_3$};
\draw(1) -- (2);
\draw(2) -- (3);
\draw(3) -- (4);
\draw(4) -- (5);
\draw(5) -- (6);
\draw(6) -- (7);
\draw(7) -- (8);
\draw(8)--(9);
\end{tikzpicture}
\end{tabular}
\end{center}

\end{example}

\begin{remark}
Following the correspondence between snake graphs and posets, we can realize Algorithm \ref{algo:New} as taking a snake graph with some internal edges weighted by positive integers and expanding this to a snake graph with unweighted edges. 

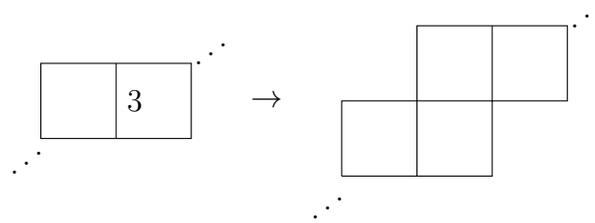
\begin{center}
\begin{tikzpicture}
 \draw(0,0) -- (2,0) -- (2,1) -- (0,1) -- (0,0);
 \draw (1,0) to node[right]{$3$} (1,1);
 \node[] at (-0.2,-0.2){$\iddots$};
 \node[] at (2.25,1.25){$\iddots$};
 \node[] at (3,0.5){$\to$};
 \draw(4,-0.5) -- (6,-0.5) -- (6,0.5) -- (7,0.5) -- (7,1.5)--(5,1.5) -- (5,0.5) -- (4,0.5) -- (4,-0.5);
 \draw(5,-0.5) -- (5,0.5) -- (6,0.5) -- (6,1.5);
  \node[] at (3.8,-0.8){$\iddots$};
 \node[] at (7.25,1.75){$\iddots$};
\end{tikzpicture}
\end{center}
\end{remark}

Next, we show an important compatibility between the posets $\calP_\gamma$ and $\calP_\gamma^\exx$.

\begin{lemma}\label{lem:ExtendPoset}
If $\calP$ be a balanced poset with associated extended poset $\calP^\exx$, then
\[
\mathcal{W}(\calP) = \mathcal{W}(\calP^\exx) = \vert J(\calP^\exx) \vert
\]
\end{lemma}

\begin{proof}
Notice the second equality is trivial since every order ideal in $\calP^\exx$ has weight 1. It suffices to show that if $\calP'$ is the result of applying one iteration of step 2 in Algorithm \ref{algo:New} to the balanced poset $\calP$, then $\mathcal{W}(\calP) = \mathcal{W}(\calP')$.

Let $i$ be the index such that we perform step 2 from Algorithm \ref{algo:New} at $\calP(i)$ and $\calP(i+1)$. Up to possibly reversing chronological ordering, we can assume $\wt(\calP(i)) = k$, so that by the balanced condition $\wt(\calP(i+1)) = \frac1k$ and $\calP(i) \prec \calP(i+1)$.

Define the following subsets of $J(\calP)$ and $J(\calP')$.

\begin{itemize}
    \item Let $\calA_0$ (resp. $\calB_0$) denote the set of order ideals $I \in J(\calP)$ ($I \in J(\calP')$) such that $\calP(i) \notin I$ ($\calP'(i) \notin I$). 
    \item Let $\calA_1$ denote the set of order ideals $I \in J(\calP)$ such that $\calP(i) \in I$ and $\calP(i+1) \notin I$.
    \item For each $1 \leq j \leq k$, let $\calB_j$ denote the set of order ideals $I \in J(\calP')$ such that $\calP'(i+j-1) \in I$ and $\calP'(i+j) \notin I$.
    \item Let $\calA_2$ (resp. $\calB_{k+1}$) denote the set of order ideals $I \in J(\calP)$ ($I \in J(\calP')$) such that $\calP(i+1) \in I$ ($\calP'(i+k) \in I$).
\end{itemize}

The sets $\calA_0,\calA_1,$ and $\calA_2$ partition $J(\calP)$ and the sets $\calB_0,\calB_1,\ldots,\calB_{k+1}$ partition $J(\calP')$. Let $\sigma: \calP \to \calP'$ be defined by \[
\sigma(\calP(j)) = \begin{cases} \calP'(j) & j \leq i \\ 
\calP'(j + k - 1) & j > i \end{cases}
\]

Extend $\sigma$ to subsets $S$ of $\calP$ by defining $\sigma(S):=\{\sigma(\rho): \rho \in S\}$.
Notice that if $j \notin \{i,i+1\}$, then $\wt(\calP(j)) = \wt(\sigma(\calP(j)))$.

We define a family of bijections between the parts of our partition of $J(\calP)$ and $J(\calP')$.
\begin{itemize}
    \item Let $\Phi_0: \calA_0 \to \calB_0$ be defined by $\Phi_0(I)=\sigma(I) := \{\sigma(\rho) : \rho \in I\}$. 
    \item For each $1 \leq j \leq k$, define $\Phi_j: \calA_1 \to \calB_j$ by $\Phi_j(I) = \sigma(I) \cup \langle \calP'(i+j-1)\rangle$. 
    \item Let $\Phi_{k+1}: \calA_2 \to \calB_{k+1}$ by defined by $\Phi_k(I) = \sigma(I) \cup \langle \calP'(i+k)\rangle$.
\end{itemize}

The map $\Phi_0$ is clearly weight-preserving. Since $\wt(\langle\calP(i+1)\rangle) = \wt(\langle\calP'(i+k)\rangle)$, $\Phi_{k+1}$ is also weight-preserving.

Fix $I \in \calA_1$ and let $W = \wt(I) /k = \wt(I \backslash \{\calP(i)\})$. For all $1 \leq j \leq k$, we have $\wt(\Phi_j(I)) = W$, so that $\sum_{j=1}^k \wt(\Phi_j(I)) = \wt(I)$.

Now, we conclude \begin{align*}
\mathcal{W}(\calP) &= \sum_{I \in J(\calP)} \wt(I) = \sum_{I \in \calA_0} \wt(I) +\sum_{I \in \calA_1} \wt(I) +\sum_{I \in \calA_2} \wt(I) \\
& = \sum_{I \in \calA_0} \wt(\Phi_0(I)) +\sum_{I \in \calA_1} \sum_{j=1}^k \wt(\Phi_j(I)) +\sum_{I \in \calA_2} \wt(\Phi_{k+1}(I))\\
&= \sum_{I \in \calB_0} \wt(I) + \sum_{j=1}^k \sum_{I \in \calB_j} \wt(I) + \sum_{I \in \calB_{k+1}} \wt(I) = \mathcal{W}(\calP').
\end{align*}
\end{proof}

\begin{remark}
One could alternatively prove Theorem \ref{thm:ComputeKMarkov} by combining Lemma \ref{lem:ExtendPoset} with the work in \cite{gyoda2024sl}, whose snake graph construction mirrors the extended posets considered here.
\end{remark}

A consequence of Lemma \ref{lem:ExtendPoset} is that the posets $\calP_\gamma^\exx$ have a nearly-palindromic symmetry. 

\begin{lemma}\label{lem:NearPalindrome}
Let $p,q$ be coprime integers. If the shape of $\calP^\exx_{\gamma_\fpq}$ is given by $a_1,\ldots,a_n$, then \begin{enumerate}
\item $n$ is even,
\item $a_i = a_{n+1-i}$ for all $1 \leq i \leq \frac{n}{2}-1$, and
\item $a_{\frac{n}{2}} = a_{\frac{n}{2}+1} \pm k$. 
\end{enumerate}
\end{lemma}

\begin{proof}
Each claim is a consequence of combining \cite[Lemma 8.21]{BG} with Lemma \ref{lem:ExtendPoset}.
\end{proof}

\section{Markov Distance}\label{sec:MarkovDistance}

We will use the construction of the posets $\calP_\gamma$ and $\calP^\exx_\gamma$ to define a distance function on $\mathbb{Z}^2$.

\begin{definition}\label{def:kMarkovLength}
\begin{enumerate}
    \item  Given an arc $\gamma$ on $\mathbb{Z}^2$, define the \emph{$k$-Markov length} of $\gamma$ by $\ell_k(\gamma) = \mathcal{W}(\calP_\gamma)$.
    \item Given $A,B \in \mathbb{Z}^2$, if $A = B$,  $\vert AB \vert_k = 0$. Otherwise, define the \emph{$k$-Markov distance} between two integral points $A,B$ to be $\vert AB\vert_k =  \min_{\delta} \ell_k(\delta)$ where we range over all (possibly generalized) arcs $\delta$ with endpoints in $A$ and $B$.
\end{enumerate}
\end{definition}

\begin{remark}
As discussed in \cite[Remark 3.7]{LLRS}, $0$-Markov distance is not a metric, and for similar it is not a metric for larger $k$.
\end{remark}

Out of convenience, we will shift all line segments $\overline{AB}$ to the origin. This will not affect our calculations.

\begin{lemma}\label{lem:Translate}
If $k$ is a nonnegative integer, $A = (q,p), B = (s,r), A' = (0,0),$ and $B' = (s-q, r-p)$, then $\vert AB \vert_k = \vert A'B'\vert_k$.
\end{lemma}

\begin{proof}
Translating  an arc $\gamma$ with endpoints at $A$ and $B$ to  $\gamma'$ with endpoints at $A'$ and $B'$ does not affect the crossing behavior, hence $\calP_{\gamma}$ and $\calP_{\gamma'}$ are the same as weighted posets. In particular, any arc $\delta$ with endpoints in $A$ and $B$ minimizing $\ell_k$ can be translated, and conversely for $\delta'$ with endpoints in $A'$ and $B'$. 
\end{proof}

 Lemma \ref{lem:Translate} allows us to safely always set $A = (0,0)$ throughout this section. Next, relating the posets $\calP_\gamma$ and $\calP^\exx_\gamma$ allows us to express $k$-Markov length in terms of continued fractions. 

\begin{cor}\label{cor:CountLength}
Let $\gamma$ be an arc on $\mathbb{Z}^2$. We have
\[
\ell_k(\gamma)  = \mathcal{N}[a_1,\ldots,a_n],
\]
where $a_1,\ldots,a_n$ is the shape of $\calP^\exx_\gamma$.
\end{cor}

\begin{proof}
This follows from combining Lemma \ref{lem:ExtendPoset}  and Corollary \ref{cor:CountOrderIdeals}.
\end{proof}

For instance, given $\gamma_1$ and $\gamma_2$ as in Example \ref{Ex:DoesNotDistinguishHomotopy}, by computing the shape of each poset and using Corollary \ref{cor:CountOrderIdeals}, we have $\ell_k(\gamma_1) = 25$ and $\ell_k(\gamma_2) = 49$.  This example points out that $k$-Markov length is not constant across the homotopy class of an arc, even when restricting to homotopic arcs with the same crossing sequence. This is a unique feature of the $k > 0$ setting.

Our goal is to study which arcs $\gamma$ with endpoints $A,B$ satisfy $\ell_k(\gamma) = \vert AB \vert_k$. When searching for an arc with a fixed pair of endpoints which minimizes $\ell_k$, we can restrict ourself to arcs which, within their homotopy class, minimize the number of intersections with line segments in $\mathcal{L}$ and the number of self-intersections. We begin by showing that such a minimizing arc cannot have any self-intersection. 

\begin{lemma}\label{lem:SelfIntNotMinimal}
If $\gamma$ is an arc with endpoints $A,B$ which has a non-contractible self-intersection, then $\ell_k(\gamma) > \vert AB \vert_k$.
\end{lemma}

\begin{proof}
The claim is trivial if $A = B$, since $\vert AA \vert_k = 0$, so we assume $A$ and $B$ are distinct. Let $\gamma$ be an arc with a point of self-intersection. Call this point $C$. Since $\gamma$ visits $C$ twice, we can divide $\gamma$ into three curves $\gamma_1,\gamma_2,\gamma_3$ consisting of the portion of $\gamma$ before the first time it passes through $C$, the portion between its two instances of passing through $C$, and the portion after the last instance.  In particular, giving each $\gamma_i$ the orientation inherited from $\gamma$, $\gamma_1$ starts at $A$ and $\gamma_3$ ends at $B$.  Let $\gamma'$ be the result of concatenating $\gamma_1$ and $\gamma_3$ and removing unnecessary intersections of the resulting arc with $\mathcal{L}$. Do this in such a way to preserve the location of the intersection points of all unaffected crossings between the $\gamma_1$ and $\mathcal{L}$ and similarly for $\gamma_3$. 

Removing the unnecessary intersections is akin to pulling $C$, the point of self-intersection, back to an extreme triangle. Assume first that $C$ cannot be pulled back to the first or last triangle which $\gamma$ passes through. Then, the configuration of $\gamma$ and $\gamma'$ is as below. 

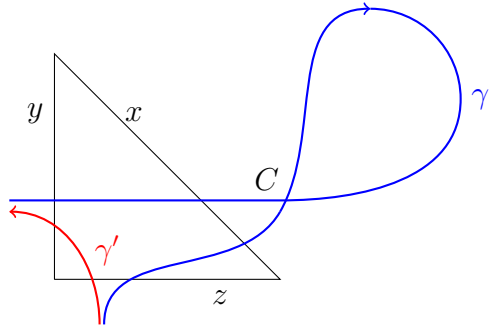
\begin{center}
\begin{tikzpicture}[scale = 3]
\draw(0,0) to  (1,0) to (0,1) to (0,0); 
\draw[red, thick,<-] (-0.2,0.3) to [out = 0, in = 90] (0.2,-0.2);
\draw[blue, thick,] (-0.2,0.35) to [out = 0, in = 180] (1,0.35) to [out = 0, in = -90] (1.8,0.8) to [out = 90, in = 0] (1.4,1.2);
\draw[blue,thick,<-] (1.4,1.2)to [out = 180, in = 60] (1,0.3);
\draw[blue,thick] (1,0.3) to [out = 240, in = 90] (0.22,-0.2);
\node[right, blue] at (1.8,0.8){$\gamma$};
\node[red, right] at (0.13,0.13){$\gamma'$};
\node[above, xshift = -5pt] at (1,0.35){$C$};
\end{tikzpicture}   
\end{center}

We will show $\ell_k(\gamma) > \ell_k(\gamma')$ by exhibiting an injection from $J(\calP_{\gamma'}^\exx)$ to $J(\calP_\gamma^\exx)$ which has a nonempty cokernel. Let $h = \vert \calP_\gamma^\exx \vert$. By our construction of $\gamma'$, there are values $u < v$ such that $\calP_{\gamma'}^\exx$ is the result of taking the induced subposet on $\calP_\gamma[1,u] \cup \calP_\gamma[v,h]$ and adding a relation between $\calP_\gamma(u)$ and $\calP_\gamma(v)$. If $\gamma$ has the orientation as above, inducing an orientation on $\gamma'$ as well, then the posets $\calP_\gamma$ and $\calP_{\gamma'}$ are as below. By abuse of notation, we will retain the chronological labeling from $\calP_\gamma^\exx$ when referencing elements of $\calP_{\gamma'}^\exx$. There are two cases to consider; first, assume that $\calP_\gamma(u+1)$ and $\calP_\gamma(v-1)$ are incomparable. 
%If $\gamma'$ and $\gamma$ are as above, then up to orientation, $\calP_\gamma(u)$ is labeled by $z$ and $\calP_\gamma(v)$ is labeled by $x$. With this choice of orientation, we add to $\calP_{\gamma'}$ the relation $\calP_\gamma(u) \succ \calP_\gamma(v)$. Moreover, $\calP^\exx_\gamma(u+1)$ and $\calP^\exx_\gamma(v-1)$ are both labeled by $y$. 

\begin{center}
\begin{tikzpicture}
\node[] at (0.5,0) {$\boxed{1}$};
\node[] at (1.5,0){$\cdots$};
\node(s) at (2.5,0){$\boxed{u}$};
\node(s+1) at (3,1){$\boxed{u+1}$};
\node[] at (4.25,1){$\cdots$};
\node(t-1) at (5.5,1){$\boxed{v-1}$};
\node(t) at (6,2){$\boxed{v}$};
\node[] at (7,2){$\cdots$};
\node[] at (8,2) {$\boxed{h}$};
\draw(s)--(s+1);
\draw(t-1)--(t);
\node[] at (4.25,-0.15){\highlight{$\calP^\exx_\gamma$}};
\draw(9.25,2.15) -- (9.25,-0.15);
\node[] at (10.5,1.5) {$\boxed{1}$};
\node[] at (11.5,1.5){$\cdots$};
\node(s') at (12.5,1.5){$\boxed{u}$};
\node(t') at (13,0.5){$\boxed{v}$};
\node[] at (14,0.5){$\cdots$};
\node[] at (15,0.5) {$\boxed{h}$};
\node[] at (11.5,-0.15){\highlight{$\calP^\exx_{\gamma'}$}};
\draw(s')--(t');
\end{tikzpicture}
\end{center}

 Every order ideal $I'$ of $\calP_{\gamma'}^\exx$ can naturally be viewed as an order ideal $I$ of $\calP_{\gamma}^\exx$ where we include the set $\langle \calP^\exx_{\gamma}(v-1)\rangle$ if $\calP^\exx_{\gamma}(v) \in I'$.Since we assumed  $\calP_\gamma(u+1)$ and $\calP_\gamma(v-1)$, this correspondence of sending $I' \in J(\calP_{\gamma'}^\exx)$ to $I \in J(\calP_{\gamma}^\exx)$ is a non-surjective injection, implying $\vert J(\calP_{\gamma'}^\exx) \vert < \vert J(\calP_{\gamma}^\exx) \vert$.

 Now, suppose that $\calP_\gamma(u+1)$ and $\calP_\gamma(v-1)$ are comparable. This only occurs if the self-intersection comes from $\gamma$ winding around a single integral point. This winding introduces a chain in $\calP_\gamma$. By either considering order ideals which contain this entire chain or which do not intersect this chain, one can similarly build an injection from $J(\calP_{\gamma'}^\exx)$ to $J(\calP_\gamma^\exx)$.

Finally, suppose that $C$ can be pulled back to the first triangle with $\gamma$ passes through. 

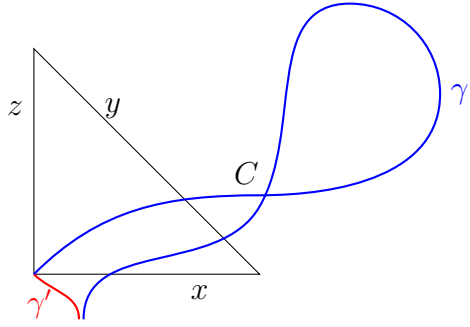
\begin{center}
\begin{tikzpicture}[scale = 3]
\draw(0,0) to (1,0) to (0,1) to  (0,0); 
\draw[red, thick] (0,0) to [out = -45, in = 90] (0.2,-0.2);
\draw[blue, thick] (0,0) to [out = 45, in = 180] (1,0.35) to [out = 0, in = -90] (1.8,0.8) to [out = 90, in = 0] (1.4,1.2) to [out = 180, in = 60] (1,0.3) to [out = 240, in = 90] (0.22,-0.2);
\node[right, blue] at (1.8,0.8){$\gamma$};
\node[red, left] at (0.13,-0.13){$\gamma'$};
\node[above, xshift = -5pt] at (1,0.35){$C$};
\end{tikzpicture}   
\end{center}

In this case, $\calP_{\gamma'}^\exx$ is of the form $\calP_{\gamma}^\exx[v,h]$ where again $h$ denotes the cardinality of $\calP_{\gamma}^\exx$. Our set-up here implies that there exists at least one value $i < v$ such that $\calP_{\gamma}^\exx(i)$ and $\calP_{\gamma}^\exx(v)$ are incomparable, and therefore it is clear that $\vert J(\calP_{\gamma'}^\exx) \vert < \vert J(\calP_{\gamma}^\exx) \vert$.
\end{proof}

\begin{remark}
The arc $\gamma'$ constructed in the proof of Lemma \ref{lem:SelfIntNotMinimal} is one of the arcs which appears when resolving the self-intersection of $\gamma$. This result could have been shown more generally by discussing skein relations on posets, as in \cite{banaian2024skein} and, in the language of snake graphs, in \cite{CS15}.  However, our construction of the poset $\calP_\gamma^\exx$ is different from the one used for cluster algebras. In order to use skein relations here, we would have to verify that every self-intersection gives rise to the correct type of configuration of the poset. Moreover, to completely discuss skein relations from self-intersections, we must associate a poset to a closed curve.  Our current proof bypasses these difficulties. 

%Looking forward, in Lemma \ref{lem:FinalStep} we are able to classify the intersections needed (with help from Proposition \ref{prop:StraightenedIsShortest}, and hence there we can use skein relations without difficulty.
\end{remark}

Our next step will be to identify an arc in each homotopy class which is minimal under $\ell_k$. Even though  this step is trivial in the $k = 0$, our description closely follows a construction within the proof of \cite[Theorem 3.5]{LLRS}. Given an arc $\gamma$ on $\mathbb{R}^2$ with endpoints $A,B \in \mathbb{Z}^2$, imagine pulling $\gamma$ tight so that $\gamma$ nearly meets a set of points in $\mathbb{Z}^2$. Call this ``straightened'' arc $\gamma^{\stir}$. The path of $\gamma^\stir$ can be broken up into a set of $r > 1$ line segments between ($\epsilon$ neighborhoods around) points in $\mathbb{Z}^2$, $A = Q_0, Q_1,\ldots,Q_r = B$, and $r-1$ angles, $\theta_1,\theta_2,\ldots,\theta_{r-1}$, which $\gamma^\stir$ takes around these points. These angles necessarily have values in $[-2\pi,-\pi] \cup [\pi,2\pi]$; otherwise, the arc could be pulled even tighter. If $\gamma$ is homotopic to $\gamma_{AB}^L$, then $\gamma^\stir = \gamma_{AB}^L$, and similarly for $\gamma_{AB}^R$. For example, the red arc, $\gamma_2$, in Example \ref{Ex:DoesNotDistinguishHomotopy} straightens to the blue arc, $\gamma_1$. 

In the following, we extend Convention \ref{conv:gammaABLR} to all arcs $\gamma^\stir$. This is necessary to make the construction of the posets $\calP_{\gamma^\stir}$ well-defined.

\begin{convention}\label{conv:StraightenedArcs}
Let $\delta$ be the $i$-th straight line segment comprising a straightened arc $\gamma^\stir$ which contains $r$ line segments total. Since $\delta$ is of the form $\gamma_\fpq$, it crosses a line from $\mathcal{L}$, say $\tau$, at its midpoint. Let $j$ be the smallest positive integer satisfying one of the following: $\theta_{i-j}$ and $\theta_{i+j}$ are the same sign; $i-j = 0$; or $i+j = r$. If $i-j = 0$ and $i+j = r$, implying that $i = \frac{r}{2}$ and for every $j'<j$, $\theta_{i-j}$ and $\theta_{i+j}$ are opposite signs, then we have $\delta$ intersect $\tau$ at exactly its midpoint. Otherwise, at least one of $\theta_{i-j}$ and $\theta_{i+j}$ is defined and each of these defined angles has the same sign. If this sign is negative (positive), we perturb $\delta$ to intersect $\tau$ slightly to the left (right) of its midpoint. \footnote{This Convention is refined from an earlier version.}
\end{convention}

\begin{figure}
    \centering
\begin{tikzpicture}[scale=1.5]
\draw (0,0) -- (4,0) -- (4,4) -- (0,4) -- (0,0);
\draw (1,0) -- (1,4);
\draw (2,0) -- (2,4);
\draw(3,0) -- (3,4);
\draw(0,1) -- (4,1);
\draw(0,2) -- (4,2);
\draw(0,3) -- (4,3);
\draw (0,1) -- (1,0);
\draw (0,2) -- (2,0);
\draw(0,3) -- (3,0);
\draw(0,4) -- (4,0);
\draw(1,4) -- (4,1);
\draw(2,4) -- (4,2);
\draw(3,4) -- (3,4);
\draw[red, thick, ->] (0,2) to (.4,2.4);
\draw[red, thick] (.4,2.4) to [out = 45, in = 180, looseness = 2] (.55,2.55) to (0.9,2.95) to [out = 90, in = 90, looseness=5] (1.1,2.95)to (1.9,2.55) to [out = 45, in = 120, looseness = 2] (2.1,2.45) to (2.9,2.05) to [out = 30, in = 0, looseness = 7](2.9,1.9) to (2.55,1.55) to [out = 300, in = 0, looseness = 2] (2.4,1.4) to (2.1,1.1) to [out = 180, in = 180, looseness = 5](2.05,0.85) to  (3,0) ;
\end{tikzpicture}
    \caption{An example of an arc $\gamma^\stir$. This arc consists of 4 line segments and three angles, $\theta_1,\theta_2$ and $\theta_3$. The first two angles are negative whereas the last is positive. We have perturbed the central intersections of the straight line segments to highlight Convention \ref{conv:StraightenedArcs}.}
    \label{fig:placeholder}
\end{figure}
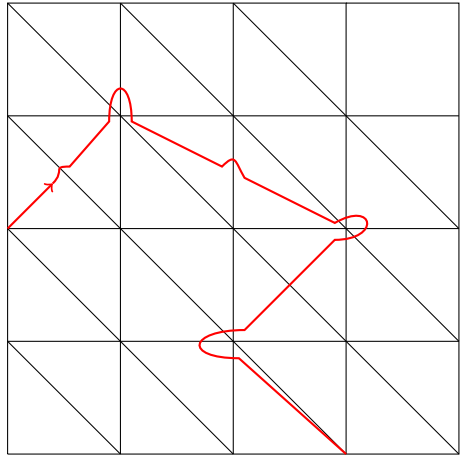

\begin{prop}\label{prop:StraightenedIsShortest}
Given $\gamma$, an arc on $\mathbb{R}^2$ with endpoints in $\mathbb{Z}^2$, we have
\[
\ell_k(\gamma) \geq \ell_k(\gamma^\stir).
\]
\end{prop}

\begin{proof}
If $\calP_\gamma = \calP_{\gamma^\stir}$, then by definition $\ell_k(\gamma) = \ell_k(\gamma^\stir)$. If $k = 0$, then the statement also holds without further proof thanks to \cite{LLRS}. 

So suppose $k > 0$ and that $\calP_\gamma \neq \calP_{\gamma^\stir}$. It suffices to consider $\gamma$ such that $\gamma$ and $\gamma^\stir$ cross the same sequence of line segments; otherwise, $\gamma$ could be pulled tight past unnecessary crossings. Assign to each poset a compatible chronological ordering and consider the first place where the posets differ with respect to this ordering. This means that there is an arc $\tau$ from $\mathcal{L}$ such that $\gamma$ crosses $\tau$ closer to its left endpoint and and $\gamma^\stir$ crosses $\tau$ closer to its right endpoint or vice versa. There are three ways in which this can happen, as are drawn below. The difference between the cases (ii) and (iii) is that, in case (ii), the intersections between $\gamma^\stir$ and the vertical lines are different types (left versus right) and in case (iii) the intersections between $\gamma^\stir$ and the vertical lines are the same type. 

\begin{center}
\begin{tikzpicture}[scale=3]
\draw (0,0) to (1,0) to node[right,yshift = 20pt]{$\tau'$} (1,1) -- (0,1) -- (0,0);
\draw(1,0) to node[right,xshift = -20pt, yshift = 20pt]{$\tau$} (0,1);
\draw[thick, blue] (0,-.2) to node[above, xshift = -10pt, yshift = -3pt]{$\gamma^\stir$} (1.1,0.5);
\draw[thick, red] (0,-.15) to [out = 45, in = 215] node[left]{$\gamma$}(0.5,0.8);
\draw[thick, red] (0.5,0.8) to [out = 35, in = 210] (1.1,0.55);
\node[] at (0.5,-0.2){$(i)$};
\draw (2,0) -- (3,0) to node[right,yshift = 30pt]{$\tau'$} (3,1) -- (2,1) -- (2,0);
\draw(3,0)to node[right,xshift = -20pt, yshift = 20pt]{$\tau$} (2,1);
\draw[thick, blue] (1.9,0.05) to node[above, xshift = -10pt, yshift = -3pt]{$\gamma^\stir$} (3.1,0.65);
\draw[thick, red] (1.9,0.15) to [out = 45, in = 215](2.5,0.8);
\draw[thick, red] (2.5,0.8) to [out = 35, in = 210] (3.1,0.7);
\node[red] at (2.9,0.8){$\gamma$};
\node[] at (2.5,-0.2){$(ii)$};
\draw (4,0) -- (5,0) to node[right,yshift = 20pt]{$\tau'$} (5,1) -- (4,1) -- (4,0);
\draw(5,0)to node[right,xshift = -20pt, yshift = 20pt]{$\tau$}(4,1);
\draw[thick, blue] (3.9,0.1) to node[above, xshift = -10pt, yshift = -3pt]{$\gamma^\stir$} (5.1,0.4);
\draw[thick, red] (3.9,0.15) to [out = 45, in = 215](4.5,0.8);
\draw[thick, red] (4.5,0.8) to [out = 35, in = 210] (5.1,0.45);
\node[red] at (4.8,0.8){$\gamma$};
\node[] at (4.5,-0.2){$(iii)$};
\end{tikzpicture}
\end{center}

Suppose first we are in Case (i). Then, the shape of $\calP_\gamma^\exx$ is of the form $a_1,\ldots,a_{i-1},k$,\\$a_{i+1},\ldots,a_n$. Applying Equation \ref{eq:ContFracSkein1} from Lemma \ref{lem:ContFracSkein}, we have

\begin{align*}
\mathcal{N}[a_1,\ldots,a_{i-1},k,a_{i+1},\ldots,a_n] &= \mathcal{N}[a_1,\ldots,a_{i-1}+k+a_{i+1},\ldots,a_n]\\
&+ k\mathcal{N}[a_1,\ldots,a_{i-1}-1,1,a_{i+1}-1,\ldots,a_n].
\end{align*}

Notice that $a_1,\ldots,a_{i-1}+k+a_{i+1},\ldots,a_n$ is the shape of the poset $\calP^\exx_{\gamma'}$ where $\gamma'$ follows $\gamma_{AB}^L$ through its intersection with $y$ then after follows $\gamma$. We have shown that $\ell_k(\gamma) \geq \ell_k(\gamma')$, and so we can continue the argument with $\gamma'$ instead. Similarly, if we are in Case (iii), then we can follow the same procedure with  Equation \ref{eq:ContFracSkein2} from Lemma \ref{lem:ContFracSkein}.

Therefore, it remains to consider only the case where $\gamma$ and $\gamma^\stir$ differ in at least one position as in Case (ii). In this case, we will instead consider the posets $\calP_\gamma$ and $\calP_{\gamma^\stir}$. We claim that in this case, $\calP_{\gamma}$ has a reverse kissing self-overlap, and that the poset $\calP_{34}$, as in Definition \ref{def:ResolveRKSO}, is equal to $\calP_{\gamma'}$ where $\gamma'$ differs from $\gamma^\stir$ in less places than $\gamma$ does. 

It is easiest if we first suppose $\gamma$ and $\gamma^\stir$ differ in exactly one position as in Case (ii). Let $i$ be such that $\calP_\gamma(i)$ and $\calP_{\gamma}(i+1)$ are the two elements labeled by $\tau$ and similarly for $\calP_{\gamma^\stir}$. Let $\delta$ be the line segment of the form $\overline{AB}$ which $\gamma^\stir$ is following at this point of difference with $\gamma$. In particular, Convention \ref{conv:StraightenedArcs} does not apply to $\delta$, and so we know that $\delta$ either intersects $\tau$ or $\tau'$ at its midpoint. 

Suppose first that $\delta$ intersects $\tau'$ at its midpoint. We  further split this case into the behavior of $\gamma^\stir$ at the endpoints of $\delta$.  For example, if $\delta$ is the first line segment of $\gamma^\stir$,  then $\calP_\gamma$ has the form $\calP_\gamma[1,i] \prec \calP_\gamma[i+1,2i] \succ \calP_\gamma(2i+1) \cdots$ while $\calP_{\gamma^\stir}$ has the form $\calP_{\gamma^\stir}[1,i] \succ \calP_{\gamma^\stir}[i+1,2i] \succ \calP_{\gamma^\stir}(2i+1) \cdots$. 
We know that in each case we have element $2i$ covers element $2i+1$ since element $2i$ is labeled by a line segment parallel to $\tau$ and element $2i+1$  is labeled by a line segment parallel to $\tau'$. 
We can check $\calP_\gamma[1,i]$ and $\calP_\gamma[i+1,2i]$ form a reverse kissing self-overlap and that, in this case where $\gamma$ and $\gamma^\stir$ only differ in one intersection, $\calP_{\gamma^\stir} = \calP_{34}$ from Definition \ref{def:ResolveRKSO}. 

If $\delta$ is not the first line segment of $\gamma^\stir$, then one can consider two further cases depending on the sign of the first angle preceding $\delta$. If the sign is negative, it is important to remember it must be no more than $-\pi$. In each case, one finds that $\calP_\gamma$ has a reverse kissing self-overlap and that $\calP_{\gamma^\stir}$ appears in its resolution. 

The other case to consider is when $\delta$ intersects $\tau$ at its midpoint. To identify the maximal self-overlap in the two posets, we must consider the integer $j$, as in Convention \ref{conv:StraightenedArcs}, which identifies the largest neighborhood around $\tau$ where the angles of $\gamma^\stir$ maintain alternating signs. As in the definition of $j$, there are a few cases based on how $\gamma^\stir$ behaves outside of this neighborhood, but in each case the same method as previously can be used. 

It $\gamma$ contains multiple such deviations from $\gamma^\stir$, we proceed by induction. If deviations are sufficiently far apart, they can be resolved independently. If there are two deviations in opposite directions, then $\gamma$ and $\gamma^\stir$ intersect and we can use part (1) of Proposition \ref{prop:SkeinRelationsOnPosets}  to replace $\gamma$ with an arc with less deviations from $\gamma^\stir$. Finally, if there are multiple deviations from Case (ii) whose overlaps interact, one can check that they combine into a larger reverse kissing self-overlap.  

\end{proof}

\begin{remark}\label{rem:MustRespectConvention}
A special case of Proposition \ref{prop:StraightenedIsShortest} is showing that if $\gamma$ is an arc which is homotopic to $\gamma^\stir$ but which does not respect Convention \ref{conv:StraightenedArcs} at one ``middle crossing'', then $\ell_k(\gamma) > \ell_k(\gamma^\stir)$.
\end{remark}

\begin{remark}
In joint work with Sen \cite{banaian2024generalization}, we discuss how the construction of $\calP_{\gamma_\fpq}$ when $k = 1$ is equivalent to the construction of a snake graph from a once-punctured sphere with three orbifold points, using \cite{BanaianKelley}. In particular, see \cite[Theorem 5]{banaian2024generalization}. In this framework, non-straightened arcs would correspond to arcs on an orbifold which have self-intersections, and the arc $\gamma'$ constructed in the proof of Proposition \ref{prop:StraightenedIsShortest} would be one arc resulting from resolving this self-intersection. 
\end{remark}

Now, we will focus on straightened arcs. Our final goal will be to show that  $\gamma_{AB}^L$ satisfies $\ell_k(\gamma_{AB}^L) = \vert AB \vert_k$ and similarly for  $\gamma_{AB}^R$. It will be helpful to record that these two arcs have the same $k$-Markov length.

\begin{lemma}\label{lem:LeftAndRightEqual}
Given $A,B \in \mathbb{Z}^2$,  the posets $\calP^\exx_{\gamma_{AB}^L}$ and $\calP^\exx_{\gamma_{AB}^R}$ are the same up to reversal. In particular, $\ell_k(\gamma_{AB}^L) = \ell_k(\gamma_{AB}^R)$.
\end{lemma}

\begin{proof}
By Lemma \ref{lem:Translate}, it suffices to set $A = (0,0)$. Let $B = (q,p)$.  If $k = 0$, then we are done by the proof of \cite[Lemma 3.2]{LLRS}. 
So suppose $k >0$. We will induct on $g:=\gcd(p,q)$.

First, suppose $g = 1$. Let $h_\fpq:= \vert \calP^\exx_{\gamma_\fpq} \vert$.  Let  $a_1,\ldots,a_n$ denote the shape of $\calP_{\gamma_\fpq}^\exx$.   
Note the only  difference between the set of cover relations of $\calP_{\gamma_{AB}^L} = \calP_{\gamma_\fpq}$ and $\calP_{\gamma_{AB}^R}$ is that we have $\calP_{\gamma_{AB}^L}(\frac12 h_\fpq) \prec  \calP_{\gamma_{AB}^L}(\frac12 h_\fpq+1)$ and $\calP_{\gamma_{AB}^R}(\frac12 h_\fpq) \succ  \calP_{\gamma_{AB}^R}(\frac12 h_\fpq+1)$ (see Convention \ref{conv:gammaABLR}).
As a consequence,  the shape of $\calP_{\gamma_{AB}^R}^\exx$ is $a_1,\ldots,a_{\frac{n}{2}-1},a_{\frac{n}{2}+1},a_{\frac{n}{2}},a_{\frac{n}{2}+2},\ldots,a_n$.  By Lemma \ref{lem:NearPalindrome},  this is equal to the reversal of $a_1,\ldots, a_n$. Therefore,  by Lemma \ref{lem:Reverse} and Corollary  \ref{cor:CountLength}, $\ell_k(\gamma_{AB}^L) = \ell_k(\gamma_{AB}^R)$ in this case.

Now, suppose we have shown our claim for all pairs of positive integers with greatest common divisor less than $g$, and suppose $\gcd(p,q) = g$.  Let $r$ and $s$ be such that $p = gr$ and $q = gs$,  and let $p' = (g-1)r$ and $q' = (g-1)s$. Set $B = (q',p')$. 
%S $\gcd(r,s) = 1$, recall the poset $\calP_{\gamma_{(0,0)(s,r)}^L}$ is equivalent to $\calP_{\gamma_\frs}$ while the base case  showed $\calP_{\gamma_{(0,0)(q,p)}^R}$ is equivalent to $\overline{\calP_{\gamma_\fpq}}$.  

The poset $\calP_{\gamma_{AB}^L}^\exx$ has $\calP_{\gamma_{AB'}^L}^\exx$ and $\calP_{\gamma_\fpq}^\exx$ as subposets,  connected as below on the left. This is true because, as in the base case, $\calP_{\gamma_{(0,0)(s,r)}^L}$ is equivalent to $\calP_{\gamma_\frs}$. The vertical dots represent a chain consisting of $3k+3$ elements which comes from traveling clockwise in a semicircle around $B'$, from $ \pi + \epsilon$ to $\epsilon$ for $\epsilon > 0$. 
Similarly, the poset $(\calP_{\gamma_{AB}^R})^\exx$ has $\calP_{\gamma_{(q,p)B}^R}^\exx$ and $\overline{\calP_{\gamma_\fpq}^\exx}$ as subposets, as below on the right. By Lemma \ref{lem:Translate},  $\calP_{\gamma_{(q,p)B}^R}^\exx$ is equivalent to $\calP_{\gamma_{AB'}^R}^\exx$. 

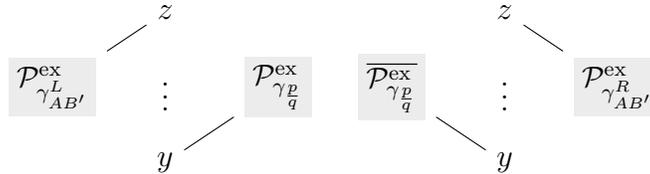
\begin{center}
\begin{tabular}{cc}
\begin{tikzpicture}
\node(s'r') at (0,0){\highlight{$\calP_{\gamma_{AB'}^L}^\exx$}};
\node(31) at (1.5,1){$\tau$};
\node(dots1) at (1.5,0){$\vdots$};
\node(21) at (1.5,-1){$\tau'$};
\node(pq) at (3,0){\highlight{$\calP_{\gamma_\fpq}^\exx$}};
\draw(s'r') -- (31);
\draw(21) -- (pq);
\end{tikzpicture}&
\begin{tikzpicture}
\node(s'r') at (0,0){\highlight{$\overline{\calP_{\gamma_\fpq}^\exx}$}};
\node(31) at (1.5,1){$\tau'$};
\node(dots1) at (1.5,0){$\vdots$};
\node(21) at (1.5,-1){$\tau$};
\node(pq) at (3,0){\highlight{$\calP_{\gamma_{AB'}^R}^\exx$}};
\draw(s'r') -- (21);
\draw(31) -- (pq);
\end{tikzpicture}\\
\end{tabular}
\end{center}

By the inductive hypothesis, $\calP_{\gamma_{AB'}^L}^\exx$ and $\calP_{\gamma_{AB'}^R}^\exx$ are related by reversal, which implies the same for $\calP_{\gamma_{AB}^L}^\exx$ and $\calP_{\gamma_{AB}^R}^\exx$. 
\end{proof}

Our final step will be showing that straightened arcs other than those of the form $\gamma_{AB}^L$ or $\gamma_{AB}^R$ do not minimize $\ell_k$. 

\begin{lemma}\label{lem:FinalStep}
Let $\gamma^\stir$ be a straightened arc with endpoints in distinct points $A,B \in \mathbb{Z}^2$. If $\gamma^\stir$ is not homotopic to $\gamma_{AB}^L$ or $\gamma_{AB}^R$, then $\ell_k(\gamma^\stir) > \ell_k(\gamma_{AB}^L)$.
\end{lemma}

\begin{proof}
Set $\gamma = \gamma^\stir$ throughout. This statement is shown for $k =0$ in the proof of \cite[Theorem 3.5]{LLRS}. Our task will be to show that we can follow the same reasoning for $k >0$. We achieve this by demonstrating that we can still apply skein relations in our setting. 

 Recall that we let $r > 0$ denote the number of straight line segments comprising $\gamma^\stir$, interspersed between $r-1$ angles $\theta_i$. If $r = 1$, then necessarily $\gamma^\stir = \gamma_\fpq = \gamma_{(0,0)(q,p)}^L$ for coprime integers $p$ and $q$, and there is nothing to check. So we will assume $r > 1$ and will assume $\gamma^\stir$ is not homotopic to $\gamma_{AB}^L$ or $\gamma_{AB}^R$.

Notice that $\gamma^\stir$ is homotopic to $\gamma_{AB}^L$ ($\gamma_{AB}^R$) if and only if every $\theta_i = -\pi$ ($\theta_i = \pi$).  Assume $A$ and $B$ are positioned such that $\gamma_{AB}^L$ has every angle $\theta_i = -\pi$. Without loss of generality, let $\gamma^\stir$ be such that $\theta_1 > 0$. Let $i$ be the smallest nonnegative integer such that $\theta_{i+1} \neq \pi$. 

In \cite{LLRS}, the authors break this into three cases based on the value of $\theta_i$: (1) $\theta_{i+1} \in [-2\pi,-\pi]$, (2) $\theta_{i+1} \in (\pi,2\pi)$, and (3) $\theta_{i+1} = 2\pi$. In each of these cases, they describe a second arc, $\delta$, which intersects $\gamma$. There will be an arc, $\gamma'$, in the resolution of this intersection which has endpoints in $A$ and $B$. Using skein relations allows them to conclude $\ell_k(\gamma) > \ell_k(\gamma')$. 

We will show how Case (1) still results in a situation for which we can apply skein relations. Notice that by our assumption that $\theta_1 > 0$, we know $i >0$. Since each $\theta_j = \pi$ for $j \leq i$, there exists a pair of coprime integers $p,q$ such that $Q_{j+1} = Q_j + (q,p)$ for all $0 \leq j \leq i$. 

Now, consider the arc $\delta$ which is the result of taking $\gamma_{AQ_{i+1}}^L$ and adjusting the intersection point which occurs in the middle between $Q_i$ and $Q_{i+1}$ so that it matches that of $\gamma$. By Proposition \ref{prop:StraightenedIsShortest}, $\ell_k(\delta) \geq \ell_k(\gamma_{AQ_{i+1}}^L)$ (see Remark \ref{rem:MustRespectConvention}). In Figure \ref{fig:Case1}, we zoom in on the behavior of these three arcs near $Q_i$ and $Q_{i+1}$ and see that $\gamma$ and $\delta$ intersect. 

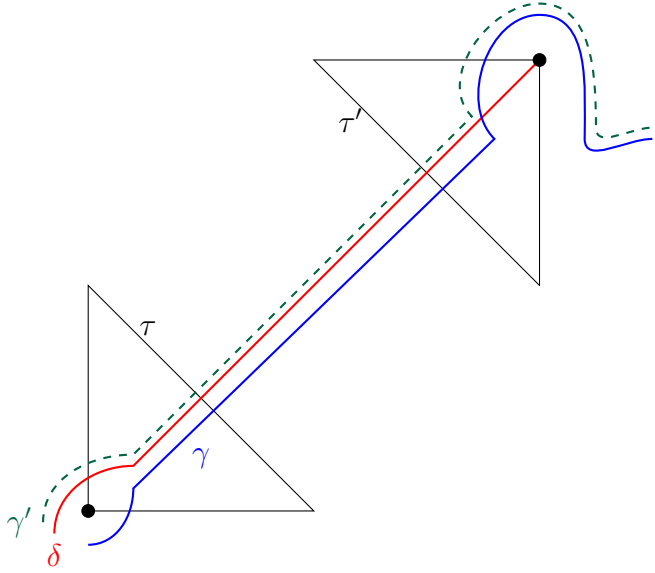
\begin{figure}
    \centering
\begin{tikzpicture}[scale=3]
\draw (0,0) -- (1,0) to node[above, xshift = -20pt, yshift = 20pt]{$\tau$} (0,1) -- (0,0);
\draw (2,1) -- (2,2) -- (1,2) to node[left, yshift = 20pt, xshift = -20pt]{$\tau'$} (2,1);
\draw[red, thick](-0.15,-0.1) to [out = 90, in = 180] (0.2,0.2) -- (2,2);
\draw[blue, thick] (0, -0.15) to [out = 0, in = 270] (0.2,0.1) to (1.8,1.65) to [out = 135, in = 180](2,2.2) to [out = 0, in =90] (2.2, 1.65) to [out = -90, in = 180] (2.5, 1.65);
\node[below, yshift = -15pt, blue] at (0.5,0.5){$\gamma$};
\draw[Dgreen, dashed, thick] (-0.2,-0.05) to [out = 90, in = 180] (0.2,0.25) -- (1.7,1.75)  to [out = 135, in = 180](2,2.25) to [out = 0, in =90] (2.25, 1.7) to [out = -90, in = 180] (2.5, 1.7);
\node[below, red] at (-0.15,-0.1){$\delta$};
\node[left, Dgreen] at (-0.2,-0.05) {$\gamma'$};
\draw[fill=black] (0,0) circle [radius=0.8pt];
\draw[fill=black] (2,2) circle [radius=0.8pt];
\end{tikzpicture}
\caption{One case of a straightened arc $\gamma^\stir = \gamma$ with endpoints in $A$ and $B$ which is not homotopic to $\gamma_{AB}^L$ or $\gamma_{AB}^R$. We intersect this arc with another, $\delta$, and resolve the intersection, yielding a shorter arc, $\gamma'$, with endpoints $A$ and $B$. See the proof of Lemma \ref{lem:FinalStep}.}
\label{fig:Case1}
\end{figure}

Let $s_1$ be the chronological label of the first element in $\calP_\gamma^\exx$ associated to the crossing of $\tau$ (as in Figure \ref{fig:Case1}) and let $t_1$ be the chronological label of the last element in this poset associated to the crossing of $\tau'$. Let $s_2$ and $t_2$ be the analogous indices in $\calP_{\delta}^\exx$; in particular, $t_2 = \vert \calP_{\delta}^\exx \vert$. By our construction of straightened arcs, we see that $\calP_\gamma^\exx[s_1,t_1]$ is isomorphic to $\calP_{\delta}^\exx[s_2,t_2]$. Here, we are using Conventions \ref{conv:gammaABLR} and the fact that $\theta_i$ and $\theta_{i+1}$ are opposite signs. Moreover, we have $\calP_\gamma^\exx(s_1-1) \succ \calP_\gamma^\exx(s_1)$, $\calP_\gamma^\exx(t_1+1) \succ \calP_\gamma^\exx(t_1)$ and $\calP_{\delta}^\exx(s_1-1) \prec \calP_{\delta}^\exx(s_1)$. Therefore, this pair of isomorphic subposets is a crossing overlap. Proposition \ref{prop:SkeinRelationsOnPosets} implies that there are four nonempty posets, $\calP_3,\calP_4,\calP_5,\calP_6$ such that \[
\vert J(\calP_\gamma^\exx) \vert \cdot \vert J(\calP_{\delta}^\exx)\vert = \vert J(\calP_3) \vert \cdot \vert  J(\calP_4)\vert +  \vert J(\calP_5) \vert \cdot \vert  J(\calP_6)\vert.
\]

Following Definition \ref{def:Type0Resolve}, $\calP_3$ is equivalent to $\calP_{\gamma'}^\exx$, where $\gamma'$ is the result of following $\delta$ from its start until its intersection point with $\gamma$ and then following $\gamma$.  Such an arc has endpoints in $A$ and $B$. The poset $\calP_4$ is associated to the arc which follows  $\gamma$ until its intersection with $\delta$ and then follows $\delta$. Call this arc $\gamma''$  and note that by Proposition \ref{prop:StraightenedIsShortest}, $\ell_k(\gamma'') \geq \ell_k(  \gamma_{AQ_{i+1}^R})$. Therefore, by combining our inequalities concerning $\delta$ and $\gamma''$ with  Proposition \ref{prop:SkeinRelationsOnPosets}, we have \[
\ell_k(\gamma) \ell_k(\delta) \geq \ell_k(\gamma) \ell_k(\gamma_{AQ_{i+1}}^L) > \ell_k(\gamma') \ell_k(\gamma'') \geq  \ell_k(\gamma') \ell_k(\gamma_{AQ_{i+1}}^R) 
\]
and by Lemma \ref{lem:LeftAndRightEqual} we conclude $\ell_k(\gamma) > \ell_k(\gamma')$.

The other cases from \cite{LLRS} will again result in crossing overlaps in the extended posets, so that the remainder of the proof readily follows. 
\end{proof}

Set $m^{(k)}_{AB} := \ell_k(\gamma_{AB}^L)$. Note that in particular, if $A = (0,0)$ and $B = (q,p)$ with $\gcd(p,q) = 1$, $m^{(k)}_{AB}$ is the $k$-Markov number $m^{(k)}_{\fpq}$. 

\begin{theorem}\label{thm:LengthIsFromPoset}
Given $A,B \in \mathbb{Z}^2$, 
\[\vert AB\vert_k = m^{(k)}_{AB}.\]
\end{theorem}

\begin{proof}

From Lemmas \ref{lem:SelfIntNotMinimal} and \ref{lem:FinalStep} and Proposition \ref{prop:StraightenedIsShortest}, we see that if $\gamma$ is any arc with endpoints in $A$ and $B$ such that $\calP_\gamma^\exx$ is not isomorphic to $\calP_{\gamma_{AB}^L}^\exx$ or $\calP_{\gamma_{AB}^R}^\exx$, then $\ell_k(\gamma) > \ell_k(\gamma_{AB}^L)$. Therefore, we conclude $\gamma_{AB}^L$ minimizes $\ell_k$, and the statement follows from the definition of $m^{(k)}_{AB}$.
\end{proof}

\section{Induced Ordering on Rational Numbers}\label{sec:InducedOrdering}

\subsection{The Generalized Aigner's Conjectures}

In \cite{LLRS}, the key tool for showing Aigner's conjectures is establishing a ``Ptolemy Inequality'' for 0-Markov distance. We will establish this for general $k$-Markov distance. 

\begin{prop}\label{prop:GeneralizedPtolemy}
If $A,B,C,D \in \mathbb{Z}^2$ are four distinct points such that the segments $\overline{AB},\overline{BC},\overline{CD}$ and $\overline{AD}$ form a convex quadrilateral with diagonals $\overline{AC}$ and $\overline{BD}$, then \[
\vert AC \vert_k \cdot \vert BD \vert_k \geq  \vert AB \vert_k \cdot \vert CD \vert_k  + \vert AD \vert_k \cdot \vert BC \vert_k 
\]
\end{prop}

The key to showing Proposition \ref{prop:GeneralizedPtolemy} is to use skein relations on crossings between arcs of the form $\gamma_{XY}^L$ and $\gamma_{XY}^R$. We introduced three variants of skein relations in Section \ref{sec:Skein}; we now give three analogous ways in which two arcs can intersect. We distinguish these by saying an intersection of Type $i$ occurs in an extreme (first or last)  triangle for $i$ of the two arcs involved.

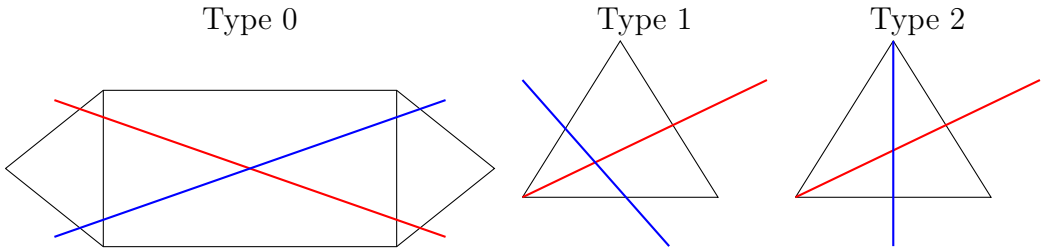
\begin{center}
\begin{tabular}{ccc}
Type 0 & Type 1 & Type 2\\
\begin{tikzpicture}[scale = 1.3]
 \draw (0,0.8) -- (1,1.6) -- (4,1.6) -- (5,0.8) -- (4,0) -- (1,0) -- (0,0.8);
 \draw (1,0) -- (1,1.6);
 \draw (4,0) -- (4,1.6);
 \draw[red,thick] (0.5,1.5) -- (4.5,0.1);
 \draw[blue,thick] (0.5,0.1) -- (4.5,1.5);
\end{tikzpicture}
&
\begin{tikzpicture}[scale = 1.3]
 \draw[] (0,0) -- (2,0) -- (1,1.6) -- (0,0); 
 \draw[red, thick] (0,0) -- (2.5,1.2);
\draw[blue,thick] (0,1.2) -- (1.5,-0.5);
\end{tikzpicture}
&
\begin{tikzpicture}[scale = 1.3]
 \draw[] (0,0) -- (2,0) -- (1,1.6) -- (0,0); 
 \draw[red, thick] (0,0) -- (2.5,1.2);
\draw[blue,thick] (1,1.6) -- (1,-0.5);
\end{tikzpicture}
\end{tabular}
\end{center}

%In the following, we write $\gamma_{AB}^{L/R}$ to denote an arc which is either $\gamma_{AB}^{L}$ or $\gamma_{AB}^{R}$.

\begin{lemma}\label{lem:Type0GoesThrough}
If two arcs $\gamma_{AC}^{L}$ and $\gamma_{BD}^{L}$ have an intersection point of Type 0, then their corresponding extended posets have a crossing overlap.
\end{lemma}

\begin{proof}
Let $\gamma_1:= \gamma_{AC}^{L}$ and $\gamma_2:= \gamma_{BD}^{L}$. When two arcs have an intersection of Type 0, they have at least one common intersection, which shows that the posets $\calP_{\gamma_1}$ and $\calP_{\gamma_2}$ will have an overlap. This further implies that $\calP_1:=\calP^\exx_{\gamma_1}$ and $\calP_2:=\calP^\exx_{\gamma_2}$ will have an overlap, say $R:=\calP_1[c,d] \cong \calP_2[c',d']$.

It remains to show that, up to relabeling, this overlap is on top in $\calP_1$ and on bottom in $\calP_2$. This can be done by analyzing a few cases near the intersections which correspond to $\calP_1(c-1)$ and $\calP_2(c'-1)$, when they exist. If $c>1$ and $c'>1$, then this will either correspond to  $\gamma_1$ and $\gamma_2$ crossing distinct arcs before the overlap or crossing the same arc, say $\tau_2$, but with one intersection point closer to the left endpoint of $\tau_2$ and the other closer to the right (using consistent orientations with respect to the overlap). 

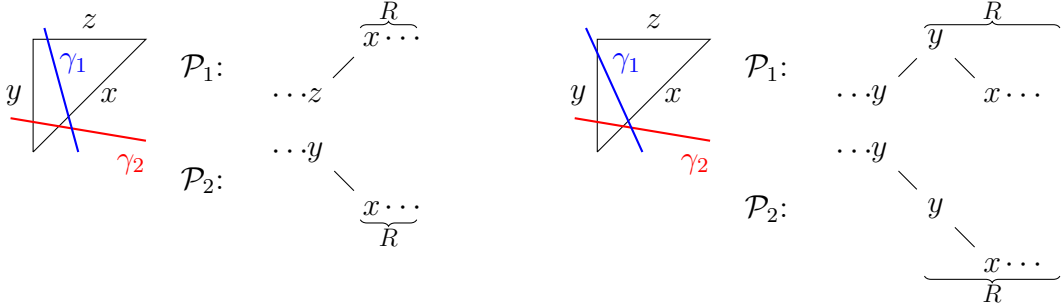
\begin{center}
\begin{tikzpicture}[scale=1.5]
\draw (0,0) to node[right]{$\tau_1$} (1,1) to node[above]{$\tau_3$} (0,1) to node[left]{$\tau_2$}(0,0);
\draw[thick, red] (-0.2,0.3) to node[below, xshift = 20pt, yshift = -5pt]{$\gamma_2$} (1,0.1);
\draw[thick, blue] (0.4,0) to node[right, yshift = 10pt, xshift = -5pt]{$\gamma_1$}(0.1,1.1);
\node[] at (1.5,0.75){$\calP_1$:};
\node() at (2.25,0.5){$\cdots$};
\node(a) at (2.5,0.5){$\tau_3$};
\node(b) at (3,1) {$\tau_1$};
\node() at (3.3,1){$\cdots$};
\draw [decorate,
    decoration = {brace}] (2.9,1.1) to node[above, scale = 0.8]{$R$} (3.4,1.1);
%\node[above, scale = 0.8] at (3.15,1.15){$R$};
\draw(a) -- (b);
\node[] at (1.5,-0.25){$\calP_2$:};
\node() at (2.25,0){$\cdots$};
\node(c) at (2.5,0){$\tau_2$};
\node(d) at (3,-0.5) {$\tau_1$};
\node() at (3.3,-0.5){$ \cdots$};
\draw(c) -- (d);
\draw [decorate,
       decoration = {brace,mirror}] (2.9,-0.6) to node[below, scale = 0.8]{$R$}  (3.4,-0.6);
\draw (5,0) to node[right]{$\tau_1$} (6,1) to node[above]{$\tau_3$} (5,1) to node[left]{$\tau_2$}(5,0);
\draw[thick, red] (4.8,0.3) to node[below, xshift = 20pt, yshift = -5pt]{$\gamma_2$} (6,0.1);
\draw[thick, blue] (5.4,0) to node[right, yshift = 10pt, xshift = -5pt]{$\gamma_1$}(4.9,1.1);
\node[] at (6.5,0.75){$\calP_1$:};
\node() at (7.25,0.5){$\cdots$};
\node(e) at (7.5,0.5){$\tau_2$};
\node(f) at (8,1){$\tau_2$};
\node(g) at (8.5,0.5){$\tau_1$};
\node() at (8.8,0.5){$\cdots$};
\draw(e) -- (f);
\draw(f) -- (g);
\draw [decorate,
    decoration = {brace}] (7.9,1.1) to node[above, scale = 0.8]{$R$}  (9.1,1.1);
\node[] at (6.5,-0.5){$\calP_2$:};
\node() at (7.25,0){$\cdots$};
\node(h) at (7.5,0){$\tau_2$};
\node(i) at (8,-0.5){$\tau_2$};
\node(j) at (8.5,-1){$\tau_1$};
\node() at (8.8,-1){$\cdots$};
\draw(h) -- (i);
\draw(i) -- (j);
\draw [decorate,
    decoration = {brace, mirror}] (7.9,-1.1) to node[below, scale = 0.8]{$R$}  (9.1,-1.1);

\end{tikzpicture}
\end{center}

If $c=1$ or $c' = 1$, this means that $\gamma_1$ or $\gamma_2$ begins ``in'' the overlap.  The same three cases are possible for $d$ and $d'$. In all such cases, one can check that the fact that $\gamma_1$ and $\gamma_2$ intersect implies that the overlap is on top in $\calP_1$ and on bottom in $\calP_2$, up to indexing. 
\end{proof}

%Given two arcs which intersect, $\gamma_1$ and $\gamma_2$, let their resolution be the pair of sets $\{\gamma_3,\gamma_4\} \cup \{\gamma_5,\gamma_6\}$. 

\begin{lemma}\label{lem:ResolutionGoesThrough}
Let $A,B,C,D \in \mathbb{Z}^2$ be four distinct points such that the segments $\overline{AB},\overline{BC},\overline{CD}$ and $\overline{AD}$ form a convex quadrilateral with diagonals $\overline{AC}$ and $\overline{BD}$. There is a resolution of the posets $\calP^\exx_{\gamma_{AC}^{L}}$ and $\calP^\exx_{\gamma_{BD}^{L}}$, say $\{\calP_3, \calP_4\} \cup \{\calP_5, \calP_6\}$ such that, up to reordering, $\calP_3 = \calP^\exx_{\gamma_3}$ for $\gamma_3$ an arc with endpoints in $A,B$ and similarly for $\calP_4$ and $C,D$; $\calP_5$ and $A,D$; and $\calP_6$ and $B,C$. 
\end{lemma}

\begin{proof}
Notice that the statement of the Lemma concerns straight line segments, such as $\overline{AC}$, and not the deformed arcs $\gamma_{AC}^{L}$. Therefore, the intersection point between $\overline{AC}$ and $\overline{BD}$ could occur at an integral point. When passing to the deformed arcs $\gamma_{AC}^{L}$ and $\gamma_{BD}^{L}$, this intersection point will move slightly off the integral point, but will remain in the  interior of the convex quadrilateral. Therefore, we can safely focus on the left-biased curves $\gamma_{AC}^L$ and $\gamma_{BD}^L$.

Assume first that  $\gamma_1:=\gamma_{AC}^{L}$ and $\gamma_2:=\gamma_{BD}^{L}$ cross in a Type 0 intersection. Orient these arcs so that they pass through the overlap in the same direction (i.e., so that the overlap has a consistent chronological labeling in the two posets).  Lemma \ref{lem:Type0GoesThrough} guarantees that the posets $\calP_1:=\calP^\exx_{\gamma_{AC}^{L}}$ and $\calP_2:=\calP_{\gamma_{BD}^{L}}$ have a crossing overlap. Comparing the resolution of a crossing overlap (Definition \ref{def:Type0Resolve}) with the construction of the poset associated to an arc (combining Algorithms \ref{algo:Old} and \ref{algo:New}), one can see that $\calP_3$ is equivalent to $\calP_{\gamma_3}$ where $\gamma_3$ is the arc given by following $\gamma_1$ until the intersection point and then following $\gamma_2$ using its established orientation. This can be denoted $\gamma_1 \circ \gamma_2$; it is important to see we read this left-to-right. In this notation, similarly $\calP_4 = \calP_{\gamma_2 \circ \gamma_1}$. Let $\gamma_5'$ be the result of following $\gamma_1$ until the intersection point, then following $\gamma_2$ backwards, and let $\gamma_5$ be the result of pulling $\gamma_5'$ tight, removing any unnecessary intersections, but preserving the locations of all (necessary) intersections. This will exactly produce an arc such that $\calP_5 = \calP_{\gamma_5}$, and the description for $\calP_6$ is similar. 

Now, suppose $\gamma_1$ and $\gamma_2$ cross in a Type 1 intersection. Without loss of generality, assume that this intersection appears in the first face which $\gamma_2$ passes through. Let $\tau_1,\tau_2,$ and $\tau_3$ be the arcs which border this triangle, such that $\gamma_2$ crosses $\tau_1$ and $\gamma_1$ crosses $\tau_2$ and $\tau_3$, in this order. In order to satisfy the chronological labeling requirement in Definition \ref{def:Type1Resolution}, assume that $\gamma_1$ $\tau_2$ and $\tau_3$ share an endpoint to the right of $\gamma_1$ (with respect to its chosen orientation).

With this structure, we can see that $\calP_3 =\calP_{\gamma_1 \circ \gamma_2}$ and $\calP_5 = \calP_{\gamma_1^{-1} \circ \gamma_2}$. Meanwhile, $\calP_4$ is associated to the arc $\gamma_4$ resulting from taking $\gamma_2 \circ \gamma_1$ and pulling this curve past all unnecessary crossings. The ``unnecessary crossings'' consist of $\tau_3$ and any arcs crossed afterwards by $\gamma_1$ which share an endpoint with $\tau_3$. The poset $\calP_6$ can be described similarly, using $\gamma_2 \circ \gamma^{-1}$.

Discussing a Type 2 intersection is similar, and we omit the details. The reason that $\calP_4$ is set to $\emptyset$ is that one arc in the geometric resolution will be homotopic to an arc in $\mathcal{L}$.
\end{proof}

With these two Lemmas, we can prove the Ptolemy inequality.

\begin{proof}[Proof of Proposition \ref{prop:GeneralizedPtolemy}]
From Theorem \ref{thm:LengthIsFromPoset}, we have $\vert AC \vert_k = \ell_k(\gamma_{AC}^L)$ and similarly for $\vert BD \vert_k$. If $\vert AC \vert_k$ and $\vert BD \vert_k$ intersect, so do $\gamma_{AC}^L$ and $\gamma_{BD}^L$. By Lemmas \ref{lem:Type0GoesThrough} and \ref{lem:ResolutionGoesThrough} and Proposition \ref{prop:SkeinRelationsOnPosets}, we have \[
\ell_k(\gamma_{AC}^L) \ell_k(\gamma_{BD}^L) = \ell_k(\gamma_3)\ell_k(\gamma_4) + \ell_k(\gamma_5)\ell_k(\gamma_6)
\]
where $\gamma_3$ has endpoints $A,B$, $\gamma_4$ has endpoints $C,D$, $\gamma_5$ has endpoints $A,D$, and $\gamma_6$ has endpoints $B,C$. Since $\ell_k(\gamma_3) \geq \vert AB \vert_k$ and similarly for the other $\gamma_i$, the claim follows. 
\end{proof}

\begin{proof}[Proof of Theorem \ref{thm:Main}]
When $k =0$, this result is \cite[Theorem 4.1]{LLRS}. The proof only uses the Ptolemy inequality. Since we have shown the Ptolemy inequality holds for general $k$, we can conclude Aigner's conjectures also hold in the $k$-Markov setting. 
\end{proof}

\subsection{More Broadly}

For each non-negative integer $k$, we have an induced partial order on $\mathbb{Q}_{[0,1]}$ given by setting $\fpq \leq_k \frs$ whenever $m^{(k)}_\fpq \leq m^{(k)}_\frs$. If $k = 0$, then claiming that $\leq_0$ is a total order is equivalent to Frobenius' conjecture. Indeed, \cite[Conjecture 1.8]{gyoda2023uniqueness} posits that each partial order $\leq_k$ is a total order (i.e., a $k$-version of the uniqueness conjecture). 

In \cite{LLRS}, the authors widely generalize Aigner's conjectures by exhibiting families of relations amongst Markov numbers whose rational labels sit on lines of various slopes. In this setting, Conjecture \ref{conj:Aigner} concerns lines of slope $0,-1$, and $\infty$.

It is natural to wonder if all these inequalities will hold for $k$-Markov numbers. The proof techniques used for the wider class of inequalities utilize more than just the Ptolemy inequality. For instance, the authors make use of the fact that the Markov numbers $m_{\frac1n}$ are the odd-indexed Fibonacci numbers and $m_{\frac{n-1}{n}}$ are odd-indexed Pell numbers. One can consider the $k$-versions of these families. These were considered for $k = 1$ in \cite[Section 4]{banaian2024generalization}. We can compare the recurrences for Markov numbers and 1-Markov numbers indexed by rational numbers $\frac1n$. The former is a well-known relation on odd-indexed Fibonacci numbers and the latter comes from \cite[Proposition 2]{banaian2024generalization}. \[
m_{\frac1n} = 3m_{\frac{1}{n-1}} - m_{\frac{1}{n-2}} \qquad \qquad m_{\frac{1}{n}}^{(1)} = 5m^{(1)}_{\frac{1}{n-1}} - m^{(1)}_{\frac{1}{n-2}} - 1.
\]

We expect that there will be similar patterns for such progressions of $k$-Markov numbers, and more generally numbers $m_{AB}^{(k)}$. \cite[Remark 6.3]{LLRS} suggests an approach using skein relations involving closed curves, which follows the spirit of this article. This would be useful in making further progress towards understanding the partial order $(\mathbb{Q}_{[0,1]},\leq_k)$. However, the slight asymmetry in the combinatorial constructions related to $k$-Markov numbers when $k > 0$ (such as in Lemma \ref{lem:NearPalindrome}) yields non-homogeneous recurrence relations as above. This could increase the level of difficulty of directly following the methods from \cite{LLRS} for $k > 0$.

These considerations leads us to believe this avenue of exploration is interesting but outside the scope of the present article. We end with a question in this direction.\footnote{This question will be answered in the affirmative in joint work with Min Huang, to appear soon.}

\begin{quest}\label{quest:AreTheyDifferent}
Do there exist $k,k' \in \mathbb{Z}_{\geq 0}$ and $\fpq,\frs \in \mathbb{Q} \cap [0,1]$ such that $\fpq \leq_k \frs$ but $\fpq >_{k'} \frs$?
\end{quest}

\bibliographystyle{abbrv}
\bibliography{bibliography}

@unpublished{BG,
title = {Cluster algebraic interpretation of generalized Markov numbers and their matrixizations},
author = {Banaian, Esther and Gyoda, Yasuaki},
eprinttype={arXiv},
eprint={2507.06900},
note={arXiv:2507.06900},
year = {2025}
}

@article {ezgieminecluster2024,
    AUTHOR = {Kantarc{\i} Oğuz, Ezgi and Y{\i}ld{\i}r{\i}m, Emine},
     TITLE = {Cluster expansions: {$T$}-walks, labeled posets and matrix
              calculations},
   JOURNAL = {J. Algebra},
  FJOURNAL = {Journal of Algebra},
    VOLUME = {669},
      YEAR = {2025},
     PAGES = {183--219},
      ISSN = {0021-8693,1090-266X},
   MRCLASS = {13F60 (05E40 06D05)},
  MRNUMBER = {4863786},
       DOI = {10.1016/j.jalgebra.2025.01.024},
       URL = {https://doi.org/10.1016/j.jalgebra.2025.01.024},
}

@article {MR4265545,
    AUTHOR = {Lagisquet, Cl\'ement and Pelantov\'a, Edita and Tavenas,
              S\'ebastien and Vuillon, Laurent},
     TITLE = {On the {M}arkov numbers: fixed numerator, denominator, and sum
              conjectures},
   JOURNAL = {Adv. in Appl. Math.},
  FJOURNAL = {Advances in Applied Mathematics},
    VOLUME = {130},
      YEAR = {2021},
     PAGES = {Paper No. 102227, 28},
      ISSN = {0196-8858,1090-2074},
   MRCLASS = {11J06 (11B39 68R15)},
  MRNUMBER = {4265545},
MRREVIEWER = {Thomas\ W.\ Cusick},
       DOI = {10.1016/j.aam.2021.102227},
       URL = {https://doi.org/10.1016/j.aam.2021.102227},
}

@article {Gaster,
    AUTHOR = {Gaster, Jonah},
     TITLE = {Boundary slopes for the {M}arkov ordering on relatively prime
              pairs},
   JOURNAL = {Adv. Math.},
  FJOURNAL = {Advances in Mathematics},
    VOLUME = {403},
      YEAR = {2022},
     PAGES = {Paper No. 108377, 15},
      ISSN = {0001-8708,1090-2082},
   MRCLASS = {11J06 (37D40)},
  MRNUMBER = {4405372},
MRREVIEWER = {Thomas\ Garrity},
       DOI = {10.1016/j.aim.2022.108377},
       URL = {https://doi.org/10.1016/j.aim.2022.108377},
}

@article{elizalde2021rowmotion,
     author = {Elizalde, Sergi and Plante, Matthew and Roby, Tom and Sagan, Bruce E.},
     title = {Rowmotion on fences},
     journal = {Algebraic Combinatorics},
     pages = {17--36},
     publisher = {The Combinatorics Consortium},
     volume = {6},
     number = {1},
     year = {2023},
     doi = {10.5802/alco.256},
     language = {en},
     url = {https://alco.centre-mersenne.org/articles/10.5802/alco.256/}
}

@article {GyodaMatsushita,
    AUTHOR = {Gyoda, Yasuaki and Matsushita, Kodai},
     TITLE = {Generalization of {M}arkov {D}iophantine equation via
              generalized cluster algebra},
   JOURNAL = {Electron. J. Combin.},
  FJOURNAL = {Electronic Journal of Combinatorics},
    VOLUME = {30},
      YEAR = {2023},
    NUMBER = {4},
     PAGES = {Paper No. 4.10, 20},
      ISSN = {1077-8926},
   MRCLASS = {11D25 (13F60)},
  MRNUMBER = {4657283},
MRREVIEWER = {Arthur\ Baragar},
       DOI = {10.37236/11420},
       URL = {https://doi.org/10.37236/11420},
}

@article {banaian2024generalization,
    AUTHOR = {Banaian, Esther and Sen, Archan},
     TITLE = {A generalization of {M}arkov numbers},
   JOURNAL = {Ramanujan J.},
  FJOURNAL = {Ramanujan Journal. An International Journal Devoted to the
              Areas of Mathematics Influenced by Ramanujan},
    VOLUME = {63},
      YEAR = {2024},
    NUMBER = {4},
     PAGES = {1021--1055},
      ISSN = {1382-4090,1572-9303},
   MRCLASS = {11D45 (05A19 05E16 11A55 13F60)},
  MRNUMBER = {4721155},
MRREVIEWER = {Hayder\ R.\ Hashim},
       DOI = {10.1007/s11139-023-00801-6},
       URL = {https://doi.org/10.1007/s11139-023-00801-6},
}

@unpublished{banaian2024skein,
    title={Skein relations for punctured surfaces},
    author={Banaian, Esther and Kang, Wonwoo and Kelley, Elizabeth},
eprinttype={arxiv},
    eprint={2409.04957},
note={arXiv:2409.04957},
    year={2024}}

@unpublished{mcshane2021convexity,
  title={Convexity and {A}igner's conjectures},
  author={McShane, Greg},
eprinttype={arxiv},
  eprint={2101.03316},
note={arXiv:2101.03316},
  year={2021}
}

@article {LLRS,
    AUTHOR = {Lee, Kyungyong and Li, Li and Rabideau, Michelle and
              Schiffler, Ralf},
     TITLE = {On the ordering of the {M}arkov numbers},
   JOURNAL = {Adv. in Appl. Math.},
  FJOURNAL = {Advances in Applied Mathematics},
    VOLUME = {143},
      YEAR = {2023},
     PAGES = {Paper No. 102453, 29},
      ISSN = {0196-8858,1090-2074},
   MRCLASS = {11J06 (11B83 11D25 13F60)},
  MRNUMBER = {4505398},
MRREVIEWER = {Thomas\ Garrity},
       DOI = {10.1016/j.aam.2022.102453},
       URL = {https://doi.org/10.1016/j.aam.2022.102453},
}

@article{rabideau2020continued,
   AUTHOR = {Rabideau, Michelle and Schiffler, Ralf},
     TITLE = {Continued fractions and orderings on the {M}arkov numbers},
   JOURNAL = {Adv. Math.},
  FJOURNAL = {Advances in Mathematics},
    VOLUME = {370},
      YEAR = {2020},
     PAGES = {107231, 18},
      ISSN = {0001-8708,1090-2082},
   MRCLASS = {11J06 (11A55 11B83 13F60 30B70)},
  MRNUMBER = {4103773},
MRREVIEWER = {Yann\ Bugeaud},
       DOI = {10.1016/j.aim.2020.107231},
       URL = {https://doi.org/10.1016/j.aim.2020.107231},
}

@article{CS13,
 AUTHOR = {Canakci, Ilke and Schiffler, Ralf},
     TITLE = {Snake graph calculus and cluster algebras from surfaces},
   JOURNAL = {J. Algebra},
  FJOURNAL = {Journal of Algebra},
    VOLUME = {382},
      YEAR = {2013},
     PAGES = {240--281},
      ISSN = {0021-8693,1090-266X},
   MRCLASS = {13F60 (05Exx)},
  MRNUMBER = {3034481},
MRREVIEWER = {Xueqing\ Chen},
       DOI = {10.1016/j.jalgebra.2013.02.018},
       URL = {https://doi.org/10.1016/j.jalgebra.2013.02.018},
}

@article{markoff1879formes,
  AUTHOR = {Markoff, A.},
     TITLE = {Sur les formes quadratiques binaires ind\'efinies},
      NOTE = {(S\'econd m\'emoire)},
   JOURNAL = {Math. Ann.},
  FJOURNAL = {Mathematische Annalen},
    VOLUME = {17},
      YEAR = {1880},
    NUMBER = {3},
     PAGES = {379--399},
      ISSN = {0025-5831,1432-1807},
   MRCLASS = {99-04},
  MRNUMBER = {1510073},
       DOI = {10.1007/BF01446234},
       URL = {https://doi.org/10.1007/BF01446234},
}

@article{fomin2002cluster,
  AUTHOR = {Fomin, Sergey and Zelevinsky, Andrei},
     TITLE = {Cluster algebras {I}: {F}oundations},
   JOURNAL = {J. Amer. Math. Soc.},
    VOLUME = {15},
      YEAR = {2002},
     PAGES = {497--529},
  MRNUMBER = {1887642},
}

@article{propp2005combinatorics,
  title="{The combinatorics of frieze patterns and Markoff numbers}",
  author={Propp, James},
 journal={INTEGERS},
  VOLUME = {20},
  PAGES = {A12},
 MRNUMBER ={4067101},
  year={2020}
}

@article{beineke2011cluster,
AUTHOR = {Beineke, Andre and Br\"ustle, Thomas and Hille, Lutz},
     TITLE = {Cluster-cyclic quivers with three vertices and the {M}arkov
              equation},
      NOTE = {With an appendix by Otto Kerner},
   JOURNAL = {Algebr. Represent. Theory},
  FJOURNAL = {Algebras and Representation Theory},
    VOLUME = {14},
      YEAR = {2011},
    NUMBER = {1},
     PAGES = {97--112},
      ISSN = {1386-923X,1572-9079},
   MRCLASS = {16G20 (13F60)},
  MRNUMBER = {2763295},
MRREVIEWER = {Gregoire\ Dupont},
       DOI = {10.1007/s10468-009-9179-9},
       URL = {https://doi.org/10.1007/s10468-009-9179-9},
}

@book{aigner2013Markov,
   AUTHOR = {Aigner, Martin},
     TITLE = {Markov's theorem and 100 years of the uniqueness conjecture},
      NOTE = {A mathematical journey from irrational numbers to perfect
              matchings},
 PUBLISHER = {Springer, Cham},
      YEAR = {2013},
     PAGES = {x+257},
      ISBN = {978-3-319-00887-5},
   MRCLASS = {11-03 (11J06 20E05 20H10 68R15)},
  MRNUMBER = {3098784},
MRREVIEWER = {Thomas\ A.\ Schmidt},
       DOI = {10.1007/978-3-319-00888-2},
       URL = {https://doi.org/10.1007/978-3-319-00888-2},
}

@article {FST-I,
    AUTHOR = {Fomin, Sergey and Shapiro, Michael and Thurston, Dylan},
     TITLE = {Cluster algebras and triangulated surfaces. {I}. {C}luster
              complexes},
   JOURNAL = {Acta Math.},
  FJOURNAL = {Acta Mathematica},
    VOLUME = {201},
      YEAR = {2008},
    NUMBER = {1},
     PAGES = {83--146},
      ISSN = {0001-5962,1871-2509},
   MRCLASS = {57Q15 (13F60 32G15 52B70)},
  MRNUMBER = {2448067},
MRREVIEWER = {Christof\ Gei\ss},
       DOI = {10.1007/s11511-008-0030-7},
       URL = {https://doi.org/10.1007/s11511-008-0030-7},
}

@article{FT-II,
    AUTHOR = {Fomin, Sergey and Thurston, Dylan},
     TITLE = {Cluster algebras and triangulated surfaces {P}art {II}:
              {L}ambda lengths},
   JOURNAL = {Mem. Amer. Math. Soc.},
  FJOURNAL = {Memoirs of the American Mathematical Society},
    VOLUME = {255},
      YEAR = {2018},
    NUMBER = {1223},
     PAGES = {v+97},
      ISSN = {0065-9266,1947-6221},
   MRCLASS = {13F60 (30F60 57M50)},
  MRNUMBER = {3852257},
MRREVIEWER = {Christof\ Gei\ss},
       DOI = {10.1090/memo/1223},
       URL = {https://doi.org/10.1090/memo/1223},
}

@article {musiker2011positivity,
    AUTHOR = {Musiker, Gregg and Schiffler, Ralf and Williams, Lauren},
     TITLE = {Positivity for cluster algebras from surfaces},
   JOURNAL = {Adv. Math.},
  FJOURNAL = {Advances in Mathematics},
    VOLUME = {227},
      YEAR = {2011},
    NUMBER = {6},
     PAGES = {2241--2308},
      ISSN = {0001-8708,1090-2082},
   MRCLASS = {13F60 (05C70 05E15)},
  MRNUMBER = {2807089},
MRREVIEWER = {Gregoire\ Dupont},
       DOI = {10.1016/j.aim.2011.04.018},
       URL = {https://doi.org/10.1016/j.aim.2011.04.018},
}

@article {CSContFrac,
    AUTHOR = {{\'I}lke {\c{C}}anak\c{c}{\i }  and Schiffler, Ralf},
     TITLE = {Cluster algebras and continued fractions},
   JOURNAL = {Compos. Math.},
  FJOURNAL = {Compositio Mathematica},
    VOLUME = {154},
      YEAR = {2018},
    NUMBER = {3},
     PAGES = {565--593},
      ISSN = {0010-437X,1570-5846},
   MRCLASS = {13F60 (11A55 30B70)},
  MRNUMBER = {3778183},
MRREVIEWER = {John\ Machacek},
       DOI = {10.1112/S0010437X17007631},
       URL = {https://doi.org/10.1112/S0010437X17007631},
}

@unpublished{gyoda2024sl,
  title={SL(2,Z)-matrixizations of generalized Markov numbers},
  author={Gyoda, Yasuaki and Maruyama, Shuhei and Sato, Yusuke},
  eprint={2407.08203},
  eprinttype={arXiv},
  note={arXiv:2407.08203},
  year={2024}
}

@unpublished{gyoda2023uniqueness,
  title={Uniqueness theorem of generalized Markov numbers that are prime powers},
  author={Gyoda, Yasuaki and Maruyama, Shuhei},
  eprinttype={arXiv},
  eprint={2312.07329},
  note={arXiv:2312.07329},
  year={2023}
}

@unpublished{pilaud2023posets,
  title={Posets for $ F $-polynomials in cluster algebras from surfaces},
  author={Pilaud, Vincent and Reading, Nathan and Schroll, Sibylle},
eprinttype={arxiv},
  eprint={2311.06033},
note={arXiv:2311.06033},
  year={2023}
}

@article {musiker2013bases,
    AUTHOR = {Musiker, Gregg and Schiffler, Ralf and Williams, Lauren},
     TITLE = {Bases for cluster algebras from surfaces},
   JOURNAL = {Compos. Math.},
  FJOURNAL = {Compositio Mathematica},
    VOLUME = {149},
      YEAR = {2013},
    NUMBER = {2},
     PAGES = {217--263},
      ISSN = {0010-437X,1570-5846},
   MRCLASS = {13F60 (05C70 05E15)},
  MRNUMBER = {3020308},
MRREVIEWER = {Olga\ Kravchenko},
       DOI = {10.1112/S0010437X12000450},
       URL = {https://doi.org/10.1112/S0010437X12000450},
}

@article {Cohn,
    AUTHOR = {Cohn, Harvey},
     TITLE = {Representation of {M}arkoff's binary quadratic forms by
              geodesics on a perforated torus},
   JOURNAL = {Acta Arith.},
  FJOURNAL = {Polska Akademia Nauk. Instytut Matematyczny. Acta Arithmetica},
    VOLUME = {18},
      YEAR = {1971},
     PAGES = {125--136},
      ISSN = {0065-1036},
   MRCLASS = {10.16},
  MRNUMBER = {288079},
MRREVIEWER = {A.\ N.\ Andrianov},
       DOI = {10.4064/aa-18-1-125-136},
       URL = {https://doi.org/10.4064/aa-18-1-125-136},
}

@book {Christoffel,
    AUTHOR = {Reutenauer, Christophe},
     TITLE = {From {C}hristoffel words to {M}arkoff numbers},
 PUBLISHER = {Oxford University Press, Oxford},
      YEAR = {2019},
     PAGES = {xi+156},
      ISBN = {978-0-19-882754-2},
   MRCLASS = {68-02 (11J06 68R15)},
  MRNUMBER = {3887697},
MRREVIEWER = {Takao\ Komatsu},
}

@article {CS15,
    AUTHOR = {{\'I}lke {\c{C}}anak\c{c}{\i } and Schiffler, Ralf},
     TITLE = {Snake graph calculus and cluster algebras from surfaces {II}:
              self-crossing snake graphs},
   JOURNAL = {Math. Z.},
  FJOURNAL = {Mathematische Zeitschrift},
    VOLUME = {281},
      YEAR = {2015},
    NUMBER = {1-2},
     PAGES = {55--102},
      ISSN = {0025-5874,1432-1823},
   MRCLASS = {13F60 (05E10)},
  MRNUMBER = {3384863},
MRREVIEWER = {Fan\ Qin},
       DOI = {10.1007/s00209-015-1475-y},
       URL = {https://doi.org/10.1007/s00209-015-1475-y},
}

@article {BanaianKelley,
    AUTHOR = {Banaian, Esther and Kelley, Elizabeth},
     TITLE = {Snake graphs from triangulated orbifolds},
   JOURNAL = {SIGMA Symmetry Integrability Geom. Methods Appl.},
  FJOURNAL = {SIGMA. Symmetry, Integrability and Geometry. Methods and
              Applications},
    VOLUME = {16},
      YEAR = {2020},
     PAGES = {Paper No. 138, 50},
      ISSN = {1815-0659},
   MRCLASS = {05E16 (05C70 13F60 16S99)},
  MRNUMBER = {4188854},
       DOI = {10.3842/SIGMA.2020.138},
       URL = {https://doi.org/10.3842/SIGMA.2020.138},
}
\end{document}